\documentclass[reqno,12pt]{amsart}

\usepackage[all]{xy}
\usepackage{amsthm, amssymb, amsmath, amscd, mathrsfs}
\usepackage{geometry}
\def\a{\alpha}
\def\b{\beta}
\def\c{\gamma}

\def\e{\epsilon}

\def\s{\sigma}
\def\t{\tau}

\def\NN{{\mathbb N}}
\def\PP{{\mathbb P}}

\def\ZZ{{\mathbb Z}}

\def\cal{\mathcal}

\def\cA{{\cal A}} 
\def\cB{{\cal B}}
\def\cC{{\mathscr C}}
\def\cD{{\mathscr D}}

\def\cH{{\mathscr  H}}
\def\cHom{{\cal Hom}}

\def\cM{{\cal M}} 
\def\cN{{\cal N}} 
\def\cO{{\cal O}} 

\def\cT{{\mathscr T}}

\def\cX{{\cal X}} 
\def\cY{{\cal Y}} 

\def\fm{{\mathfrak m}}

\def\M{{\mathfrak M}}

\def\cd{\operatorname {cd}}

\def\coker{\operatorname {coker}}

\def\End{\operatorname {E nd}}
\def\Ext{\operatorname {Ext}}

\def\Hom{\operatorname {Hom}}
\def\id{\operatorname {id}}

\def\Aut{\operatorname{Aut}}

\def\cd{\operatorname{cd}}

\def\Coker{\operatorname{Coker}}

\def\depth{\operatorname{depth}}

\def\dim{\operatorname{dim}}

\def\End{\operatorname{End}}

\def\Ext{\operatorname{Ext}}

\def\uExt{\operatorname{\underline{Ext}}}

\def\gldim{\operatorname{gldim}}
\def\grmod{\operatorname{grmod}}
\def\grproj{\operatorname{grproj}}

\def\GrMod{\operatorname{GrMod}}
\def\H{\operatorname{H}}
\def\Hom{\operatorname{Hom}}

\def\RHom{\operatorname{{\bf R}Hom}}

\def\uHom{\operatorname{\underline{Hom}}}

\def\injdim{\operatorname{injdim}}

\def\lcd{\operatorname{lcd}}

\def\mod{\operatorname{mod}}
\def\Mod{\operatorname{Mod}}

\def\Proj{\operatorname{Proj}}

\def\sup{\operatorname{sup}}

\def\tails{\operatorname{tails}}

\def\tors{\operatorname{tors}}

\def\coh{\operatorname{coh}}

\def\uExt{\operatorname{\underline{Ext}}}

\def\uH{\operatorname{\underline{H}}}
\def\uHom{\operatorname{\underline{Hom}}}

\def\RuHom{\operatorname{{\bf R}\underline{Hom}}}

\newcommand{\lotimes}{\otimes^{\bf{L}}}

\def\<{\langle}
\def\>{\rangle}
\def\CM{\operatorname {CM}}
\def\uCM{\underline {\operatorname
{CM}}}
\def\GrAut{\operatorname{GrAut}}

\let\oldtext\text
\def\text#1{\oldtext{\normalshape #1}}

\def\Ga{\Gamma}

\def\RuGa{\operatorname{{\bf R}\underline {\Ga }}_\fm}

\newcommand{\uGa}{\operatorname{\underline{\Ga}}\nolimits}

\def\rnum#1{\expandafter{\romannumeral #1}}
\def\Rnum#1{\uppercase\expandafter{\romannumeral #1}}

\theoremstyle{plain} 
\newtheorem{theorem}{Theorem}[section]
\newtheorem{corollary}[theorem]{Corollary}
\newtheorem{lemma}[theorem]{Lemma}
\newtheorem{proposition}[theorem]{Proposition}

\theoremstyle{definition}
\newtheorem{definition}[theorem]{Definition}
\newtheorem{example}[theorem]{Example}
\newtheorem{question}[theorem]{Question}

\theoremstyle{remark}

\newtheorem{remark}[theorem]{Remark}

\numberwithin{equation}{section}

\begin{document}

\pagenumbering{arabic}

\title{A Categorical Characterization of Quantum Projective Spaces}


\author{Izuru Mori}

\address{
Department of Mathematics,
Faculty of Science,
Shizuoka University,
836 Ohya, Suruga-ku, Shizuoka 422-8529, Japan}

\email{mori.izuru@shizuoka.ac.jp}

\author{Kenta Ueyama}

\address{
Department of Mathematics,
Faculty of Education,
Hirosaki University,
1 Bunkyocho, Hirosaki, Aomori 036-8560, Japan}

\email{k-ueyama@hirosaki-u.ac.jp}

\thanks {
The first author was supported by JSPS Grant-in-Aid for Scientific Research (C) 16K05097
and JSPS Grant-in-Aid for Scientific Research (B) 16H03923.
The second author was supported by JSPS Grant-in-Aid for Young Scientists (B) 15K17503.}



\begin{abstract}
Let $R$ be a finite dimensional algebra of finite global dimension over a field $k$.  In this paper, we will characterize a $k$-linear abelian category $\mathscr C$ such that $\mathscr C\cong \operatorname {tails} A$ for some graded right coherent AS-regular algebra $A$ over $R$.
As an application, we will prove that if $\mathscr C$ is a smooth quadric surface in a quantum $\mathbb P^3$ in the sense of Smith and Van den Bergh, then there exists a right noetherian AS-regular algebra $A$ over $kK_2$ of dimension 3 and of Gorenstein parameter 2 such that $\mathscr C\cong \operatorname {tails} A$ where $kK_2$ is the path algebra of the 2-Kronecker quiver $K_2$.  
\end{abstract} 

\maketitle

\tableofcontents

\section{Introduction} 

\subsection{Motivation}
In 1955, J.-P. Serre \cite{Se} introduced and studied the cohomology groups of coherent sheaves
on projective schemes. 
In particular, he proved that the category $\coh (\Proj A)$ of coherent sheaves on the projective scheme $\Proj A$ is equivalent to the category of finitely generated graded $A$-modules modulo finite dimensional modules.
In 1994, motivated by Serre's work, M. Artin and J. J. Zhang \cite{AZ} introduced the categorical notion of a noncommutative projective scheme, and established a fundamental theory of noncommutative projective schemes.
Since then, the study of noncommutative projective schemes has been one of the major projects in noncommutative algebraic geometry.

The noncommutative projective scheme associated to an AS-regular algebra of dimension $n+1$ is considered as
a quantum projective space of dimension $n$.
Since projective spaces are the most basic and important class of projective schemes in commutative algebraic geometry,
quantum projective spaces have been studied deeply and extensively in noncommutative algebraic geometry.
It is known that although quantum projective spaces have many nice properties in common with $\coh \PP^n$, their structures vary widely.
In this paper, we consider the following question.

\begin{question} \label{ques.intro}
Fix a field $k$. When is a given $k$-linear abelian category $\cC$ equivalent to a quantum projective space?
That is, can we find necessary and sufficient conditions on a $k$-linear abelian category $\cC$
such that $\cC$ is equivalent to the noncommutative projective scheme associated to some AS-regular algebra?
\end{question}

If a $k$-linear abelian category $\cC$ is equivalent to a quantum projective space,
then we can investigate $\cC$ using the rich techniques of noncommutative algebraic geometry.
In this sense, Question \ref{ques.intro} is important. 
The following is the main result of this paper, which gives a complete answer to Question \ref{ques.intro}.

\begin{theorem} [Theorem \ref{thm.main}]
Let $R$ be a finite dimensional algebra of finite global dimension over a field $k$.  
Then a $k$-linear abelian category $\cC$ is equivalent to the noncommutative projective scheme associated to some AS-regular algebra $A$ over $A_0\cong R$ of Gorenstein parameter $\ell$ if and only if
\begin{enumerate}
\item[(AS1)] $\cC$ has a canonical bimodule $\omega_{\cC}$, and
\item[(AS2)] there exist an object $\cO\in \cC$ and a $k$-linear autoequivalence $s\in \Aut _k\cC$ such that 
\begin{enumerate}
\item{} $(\cO, s)$ is ample for $\cC$ (in the sense of Artin and Zhang \cite{AZ}), 
\item{} $\{s^i\cO\}_{i\in \ZZ}$ is a full geometric relative helix of period $\ell$ for $\cD^b(\cC)$,
and 
\item{} $\End _{\cC}(\cO)\cong R$.
\end{enumerate}
\end{enumerate}
\end{theorem}
Roughly speaking, (AS1) requires that $\cC$ has an autoequivalence which induces a Serre functor for $\cD^b(\cC)$ (Definition \ref{dfn.cf}),
and (AS2)(b) requires that $\cD^b(\cC)$ has a ``relaxed'' version of a full geometric helix, consisting of shifts of a single object in $\cC$  (Definitions \ref{dfn.res}, \ref{dfn.rh}).

In the last section, we will give an application of the main result.
It is well-known that if $Q$ is a smooth quadric surface in $\PP^3$,
then there exists a noetherian AS-regular algebra $A=k\<x, y\>/(x^2y-yx^2, xy^2-y^2x)$ of dimension 3
such that $\coh Q$ is equivalent to the noncommutative projective scheme associated to $A$.
Using our main result, we will prove a noncommutative generalization of this result. 
Namely we will show that if $\cC$ is a smooth quadric surface in a quantum $\PP^3$ in the sense of Smith and Van den Bergh \cite{SV},
then there exists a right noetherian AS-regular algebra $A$ over $kK_2$ of dimension 3 such that
$\cC$ is equivalent to the noncommutative projective scheme associated to $A$ where $kK_2$ is the path algebra of the 2-Kronecker quiver $K_2$
(Theorem \ref{thm.sq}).

\subsection{Notation} 
In this subsection, we introduce some notation and terminology that will be used in this paper.
Throughout, let $k$ be a field. We assume that all algebras are over $k$.
For an algebra $R$, we denote by $\Mod R$ the category of right $R$-modules,
and by $\mod R$ the full subcategory consisting of finitely presented right $R$-modules.
Note that if $R$ is a finite dimensional algebra,
then $\mod R$ is simply the full subcategory consisting of finite dimensional $R$-modules.
We denote by $R^o$ the opposite algebra of $R$ and define $R^e := R^o \otimes_k R$.
For algebras $R,S$, $\Mod R^o$ is identified with the category of left $R$-modules,
and $\Mod(R^o \otimes_k S)$ is identified with the category of $R$-$S$ bimodules,
so that $\Mod R^e$ is identified with the category of $R$-$R$ bimodules.

For a vector space $V$ over $k$,
we denote by $DV = \Hom_k(V, k)$ the vector space dual of $V$ over $k$.
By abuse of notation, for a graded vector space $V=\bigoplus_{i \in \ZZ} V_i$,
we denote by $DV$ the graded vector space dual of $V$ defined as $(DV)_i =D(V_{-i})$ for $i \in \ZZ$.
We say that a graded vector space $V$ is locally finite if $\dim_k V_i < \infty$ for all $i \in \ZZ$.
In this case, we define the Hilbert series of $V$ by
$H_V (t) := \sum_{i \in \ZZ} (\dim_k V_i)t^i \in \ZZ[[t, t^{-1}]].$

In this paper, a graded algebra means a $\ZZ$-graded algebra over a field $k$,
although we mainly deal with $\NN$-graded algebras.
For a graded algebra $A$, we denote by $\GrMod A$ the category of graded right $A$-modules,
and by $\grmod A$ the full subcategory consisting of finitely presented graded right $A$-modules.
Morphisms in $\GrMod A$ are $A$-module homomorphisms preserving degrees.
For $M \in \GrMod A$ and a graded algebra automorphism $\s$ of $A$, 
we define the twist $M_\s \in \GrMod A$ by $M_\s = M$ as a graded $k$-vector space with the new right action
$m * a = m\s (a)$.

Let $A$ be a $\ZZ$-graded algebra, and $r\in \NN^+$.  
The $r$-th Veronese algebra of $A$ is defined by 
$$A^{(r)}:=\bigoplus _{i\in \ZZ}A_{ri},$$ 
and the 
$r$-th quasi-Veronese algebra of $A$ is defined by 
$$A^{[r]}:= \bigoplus_{i \in \ZZ} \begin{pmatrix}
A_{ri} &A_{ri+1} & \cdots & A_{ri+r-1} \\
A_{ri-1} & A_{ri} & \cdots & A_{ri+r-2} \\
\vdots & \vdots & \ddots & \vdots \\
A_{ri-r+1} & A_{ri-r+2} & \cdots & A_{ri}\end{pmatrix}$$
where the
multiplication of $A^{[r]}$ is given by
$(a_{ij})(b_{ij})=(\sum_ka_{kj}b_{ik})$ (see \cite {Mbc}).
There exists an
equivalence functor $Q:\GrMod A\to \GrMod A^{[r]}$ defined by
$$Q(M)= \bigoplus_{i \in \ZZ}
\begin{pmatrix}
M_{ri} \\ M_{ri-1} \\ \vdots \\M_{ri-r+1}
\end{pmatrix}$$
where the right action of $A^{[r]}$ on $Q(M)$ is
given by $(m_i)(a_{ij})=(\sum _km_ka_{ik})$ (see \cite {Mbc}).  

Let $A = \bigoplus_{i \in \NN} A_i$ be an $\NN$-graded algebra.
We say that $A$ is connected graded if $A_0 = k$.
For a graded module $M \in \GrMod A$ and an integer $n \in \ZZ$,
we define the truncation $M_{\geq n} := \bigoplus_{i \geq n} M_i \in \GrMod A$
and the shift $M(n) \in \GrMod A$ by $M(n)_i := M_{n+i}$ for $i \in \ZZ$.
The rule $M \mapsto  M(n)$ is a $k$-linear autoequivalence for $\GrMod A$,
called the shift functor. For $M, N \in \GrMod A$, we write the vector space
$\Ext^i_A(M, N) := \Ext^i_{\GrMod A}(M, N)$ and the graded vector space
$$\uExt^i_A(M, N) := \bigoplus_{n\in\ZZ} \Ext^i_A(M, N(n)).$$

Let $A, C$ be $\NN$-graded algebras. 
Then $C^o\otimes _kA$ becomes an $\NN$-graded algebra by setting $(C^o\otimes _kA)_n := \bigoplus_{i+j=n} C^o_i \otimes _k A_j$.
We define the left exact functor  $\uGa_\fm : \GrMod (C^o\otimes _kA) \to \GrMod (C^o\otimes _kA)$ by
$$\uGa_\fm (M) := \lim_{n \to \infty} \uHom_A(A/A_{\geq n}, M)$$
where $\fm = A_{\geq 1}$.
The derived functor of $\uGa_\fm$ is denoted by $\RuGa$, and its cohomologies are denoted by
$\uH^i_{\fm}(M) := h^i(\RuGa(M))$.
For $M \in \GrMod A$, the depth of $M$ is defined to be
$\depth M :=\inf\{i \mid \uH_\fm^i(M) \neq 0\}$.
The local cohomological dimension of $M$ is defined to be
$\lcd M :=\sup\{i \mid \uH_\fm^i(M) \neq 0 \}$.
The cohomological dimension of $\uGa_\fm$ is defined by
$$ \cd(\uGa_\fm) := \sup \{ \lcd M \mid M \in \GrMod A \}. $$
We say that $A$ has finite cohomological dimension if $\cd(\uGa_\fm)<\infty$.  Note that if $A$ has finite global dimension, then it has finite cohomological dimension.  

For an abelian category $\cC$, we denote by $\cD(\cC)$ the derived category of $\cC$ and by $\cD^b(\cC)$ the bounded derived category of $\cC$.
For $\cM, \cN\in \cD(\cC)$, we often write $\Hom_{\cC}(\cM, \cN):=\Hom_{\cD(\cC)}(\cM, \cN)$ by abuse of notation.
For $\cM, \cN\in \cD(\cC)$ and $i \in \ZZ$, we set $\Ext^i_{\cC}(\cM, \cN)=\Hom_{\cD(\cC)}(\cM, \cN[i])$.

Connected graded AS-regular algebras defined below are the most important class of algebras in noncommutative algebraic geometry.
\begin{definition} \label{def.AS}
A locally finite connected graded algebra $A$ is called 
AS-regular (resp. AS-Gorenstein) of dimension $d$ and of Gorenstein parameter $\ell$ if the following conditions are satisfied: 
\begin{enumerate}
\item{} $\gldim A=d<\infty$ (resp. $\injdim_A A = \injdim_{A^o}A =d<\infty$), and
\item{} $\RuHom_A(k, A) \cong \RuHom_{A^o}(k, A) \cong k(\ell)[-d]$ in $\cD(\GrMod k)$. 
\end{enumerate}
\end{definition}

It is well-known that if $A$ is a noetherian AS-Gorenstein algebra of dimension $d$ and of Gorenstein parameter $\ell$,
then $A$ has a balanced dualizing complex $D\RuGa(A) \cong A_{\nu}(-\ell)[d]$ in $\cD(\GrMod A^e)$
with some graded algebra automorphism $\nu$ of $A$ (see \cite {Ye}).
This graded algebra automorphism $\nu$ is called the (generalized) Nakayama automorphism of $A$.
The graded $A$-$A$ bimodule $\omega_A := A_\nu (-\ell) \in \GrMod A^e$ is called the canonical module over $A$.

Let us recall the definition of graded coherentness.
\begin{definition}
\begin{enumerate}
\item A graded right $A$-module $M$ is called graded right coherent if it is finitely generated and every finitely generated graded submodule of $M$ is finitely presented over $A$.
\item A locally finite $\NN$-graded algebra $A$ is called graded right coherent if  $A$ and $A/A_{\geq 1}$ are graded right coherent modules.
\end{enumerate}
\end{definition}

Let $A$ be a graded right coherent algebra.  Then a graded right $A$-module is finitely presented if and only if it is graded right coherent.
In this case, $\grmod A$ is an abelian category.

\begin{proposition}  [{cf. \cite [Proposition 1.9]{Ye}}] \label{prop.coh}
If $A$ is a graded right coherent algebra, then every finite dimensional graded right $A$-module is graded right coherent.
\end{proposition}

\begin{proof}
If $S$ is a graded simple right $A$-module, then there exists a surjection $A/A_{\geq 1}(j) \to S$ for some $j\in \ZZ$.
Let $K$ be the kernel of this map. Since it 
is finite dimensional, it is a finitely generated submodule of $A/A_{\geq 1}(j)$.
Since $A/A_{\geq 1}(j)$ is graded right coherent, $K$ is finitely presented.
Since $\grmod A$ is an abelian category, $S$ is graded right coherent.
Since every finite dimensional module is a finite extension of graded simple right $A$-modules, the result follows.
\end{proof}

Let $A$ be a graded right coherent algebra. We denote by $\tors A$ the full subcategory of $\grmod A$ consisting of finite dimensional modules.
By Proposition \ref{prop.coh}, $\tors A$ is a Serre subcategory of $\grmod A$, so the quotient category
$$\tails A:= \grmod A/\tors A$$
is an abelian category.
If $A$ is a commutative graded algebra finitely generated in degree $1$ over $k$, then
$\tails A$ is equivalent to the category $\coh (\Proj A)$ of coherent sheaves on the projective scheme $\Proj A$
by Serre's theorem \cite{Se}. For this reason, $\tails A$ is called the noncommutative projective scheme associated to $A$ (see \cite{AZ} for details).

The quotient functor is denoted by $\pi: \grmod A \to \tails A$. We usually denote by $\cM = \pi M \in \tails A$ the image of $M \in \grmod A$.
Note that the $k$-linear autoequivalence $M \to M(n)$ preserves torsion modules,
so it induces a $k$-linear autoequivalence $\cM \to \cM(n)$ for $\tails A$, again called the shift functor.
For $\cM, \cN \in \tails A$, we write the vector space $\Ext^i_\cA(\cM, \cN) := \Ext^i_{\tails A}(\cM, \cN)$
and the graded vector space
$$\uExt^i_{\cA}(\cM, \cN) := \bigoplus_{n \in \ZZ} \Ext^i_\cA(\cM, \cN(n))$$
as before. 

For an abelian category $\cC$, we define the global dimension of $\cC$ by
$$ \gldim \cC := \sup\{i \mid \Ext^i_\cC(\cM, \cN) \neq 0 \; \textnormal{for some} \; \cM, \cN \in \cC \}. $$
The notion of graded isolated singularity for a noncommutative connected graded algebra $A$ has been defined using the noncommutative projective scheme $\tails A$
(see {\cite{U}, \cite{Jcm}).
\begin{definition}
A graded right coherent connected graded algebra $A$ is called a graded isolated singularity if $\gldim (\tails A) < \infty$.
\end{definition}

\section{Preliminaries} 

\subsection{Ampleness} 

The ampleness of a line bundle is essential to construct a homogeneous coordinate ring of a  projective scheme in commutative algebraic geometry.  We will define a notion of ampleness in noncommutative algebraic geometry.  

\begin{definition}
\begin{enumerate}
\item An algebraic triple consists of a $k$-linear category $\cC$, an object $\cO\in \cC$, and a $k$-linear autoequivalence $s\in \Aut _k\cC$.  In this case, we also say that $(\cO, s)$ is an algebraic pair for $\cC$.
\item A morphism of algebraic triples $(F, \theta, \mu):(\cC, \cO, s)\to (\cC', \cO', s')$ consists of a $k$-linear functor $F:\cC\to \cC'$, an isomorphism $\theta :F(\cO)\to \cO'$ and a natural transformation $\mu :F\circ s\to s'\circ F$.
\item Two algebraic triples $(\cC, \cO, s)$ and $(\cC', \cO', s')$ are isomorphic, denoted by 
$(\cC, \cO, s)\cong (\cC', \cO', s')$
if there exists a morphism of algebraic triples $(F, \theta, \mu):(\cC, \cO, s)\to (\cC', \cO', s')$ such that $F$ is an equivalence functor and $\mu$ is a natural isomorphism.
\item Two algebraic triples $(\cC, \cO, s)$ and $(\cC', \cO', s')$ are equivalent, denoted by
$(\cC, \cO, s)\sim (\cC', \cO', s')$
if there exists an equivalence functor $F:\cC\to \cC'$ such that $F(s^i\cO)\cong (s')^i\cO'$ for all $i\in \ZZ$.  
\item For an algebraic triple $(\cC, \cO, s)$, we define a graded algebra by 
$$B(\cC, \cO, s):=\bigoplus _{i\in \ZZ}\Hom_{\cC}(\cO, s^i\cO)$$
where the multiplication is given by the following rule: for $\a\in B(\cC,\cO, s)_i=\Hom_{\cC}(\cO, s^i\cO)$ and $\b\in B(\cC, \cO, s)_j=\Hom_{\cC}(\cO, s^j\cO)$, we define 
$\a\b:=s^j(\a)\circ \b\in \Hom_{\cC}(\cO, s^{i+j}\cO)=B(\cC, \cO, s)_{i+j}$.
\item For an object $\cM$ in $\cC$, we define a graded right $B(\cC, \cO, s)$-module
$$\uH^q(\cM):=\bigoplus _{i\in \ZZ}\Ext_{\cC}^q(\cO, s^i\cM)=\bigoplus _{i\in \ZZ}\Hom_{\cC}(\cO, s^i\cM[q])$$
where the right action is given by the following rule: for $\a\in \uH^q(\cM)_i=\Hom_{\cC}(\cO, s^i\cM[q])$ and $\b\in B(\cC, \cO, s)_j=\Hom_{\cC}(\cO, s^j\cO)$, we define $\a\b:=s^j(\a)\circ \b\in \Hom_{\cC}(\cO, s^{i+j}\cM[q])=\uH^q(\cM)_{i+j}$.
\item We define a graded left $B(\cC, \cO, s)$-module structure on 
$$\uH^q(\cO) =\bigoplus _{i\in \ZZ}\Hom_{\cC}(\cO, s^i\cO[q]) = \bigoplus _{i\in \ZZ}\Hom_{\cC}(\cO[-q], s^i\cO)$$
by the following rule:
for $\a \in B(\cC, \cO, s)_i=\Hom_{\cC}(\cO, s^i\cO)$ and $\b\in \uH^q(\cO)_j=\Hom_{\cC}(\cO[-q], s^j\cO)$,
we define $\a\b:=s^j(\a)\circ \b \in \Hom_{\cC}(\cO[-q], s^{i+j}\cO)=\uH^q(\cO)_{i+j}$.
\end{enumerate}
\end{definition} 

\begin{example} For an algebraic triple $(\cC, \cO, s)$ and $r\in \NN^+$, if $A=B(\cC, \cO, s)$, then $A^{(r)}\cong B(\cC, \cO, s^r)$ and $A^{[r]}\cong B(\cC, \bigoplus _{i=0}^{r-1}s^i\cO, s^r)$ (see \cite {Mbc}). 
\end{example}

\begin{remark} \label{rem.mat} 
A morphism of algebraic triples $(F, \theta, \mu):(\cC, \cO, s)\to (\cC', \cO', s')$ induces a map 
\begin{align*}
B(\cC, \cO, s)_i&=\Hom_{\cC}(\cO, s^i\cO) 
\to \Hom_{\cC'}(F(\cO), F(s^i\cO)) 
\to \Hom_{\cC'}(\cO', (s')^i\cO')=B(\cC', \cO', s')_i
\end{align*}
for every $i\in \ZZ$, which induces a graded algebra homomorphism $B(\cC, \cO, s)\to B(\cC', \cO', s')$.
In particular, if $(\cC, \cO, s)\cong (\cC', \cO', s')$, then $B(\cC, \cO, s)\cong B(\cC', \cO', s')$ as graded algebras.  
\end{remark}

\begin{example} \label{ex.phi}
If $A$ is a graded right coherent algebra, then $\pi :\grmod A\to \tails A$ induces a morphism of algebraic triples $(\grmod A, A, (1))\to (\tails A, \cA, (1))$, which induces a graded algebra homomorphism 
$$\phi _A:A\cong B(\grmod A, A, (1))\to B(\tails A, \cA, (1)).$$
Moreover, for $M\in \grmod A$, we have a graded right $A$-module homomorphism
$$\phi _M:M\cong \uH^0(M)\to \uH^0(\cM)$$
where we view $\uH^0(\cM)$ as a graded right $A$-module via $\phi_A$. 
\end{example} 

The following notion of ampleness introduced in \cite{AZ} is a key concept in noncommutative projective geometry.  

\begin{definition}    
We say that an algebraic pair $(\cO, s)$ for a $k$-linear abelian category $\cC$ is ample if 
\begin{enumerate}
\item[(A1)] for every $\cM\in \cC$, there exists an epimorphism $\bigoplus _{j=1}^ps^{-i_j}\cO\to \cM$ in $\cC$ for some $i_1, \dots, i_p\geq 0$, and  
\item[(A2)] for every epimorphism $\phi:\cM\to \cN$ in $\cC$, there exists $m\in \ZZ$ such that 
$$\Hom_{\cC}(s^{-i}\cO, \phi):\Hom_{\cC}(s^{-i}\cO, \cM)\to \Hom_{\cC}(s^{-i}\cO, \cN)$$
is surjective for every $i\geq m$.  
\end{enumerate}
\end{definition} 

A $k$-linear category $\cC$ is called $\Hom$-finite if $\dim _k\Hom_{\cC}(\cM, \cN)<\infty$ for every $\cM, \cN\in \cC$.  
The following theorem is a ``coherent'' version of \cite[Theorem 4.5]{AZ}.

\begin{theorem} \label{thm.AZ} 
The following statements hold.
\begin{enumerate}
\item 
Let $A$ be a graded algebra. If
\begin{enumerate}
\item $A$ is graded right coherent, and
\item for any $\cM \in \tails A$ and any $n \in \ZZ$, 
$\uH^0(\cM)_{\geq n}$ is graded right coherent,
\end{enumerate}
then $\tails A$ is $\Hom$-finite $k$-linear abelian category and $(\cA, (1))$ is an ample pair for $\tails A$. 

\item 
Conversely, if $(\cO, s)$ is an ample pair for a $\Hom$-finite $k$-linear abelian category $\cC$, then
\begin{enumerate}
\item $A := B(\cC, \cO, s)_{\geq 0}$ is a graded right coherent algebra,
\item for any $\cM \in \cC$ and any $n \in \ZZ$, $\uH^0(\cM)_{\geq n}$
is graded right coherent, and
\item the functor $\cC \to \tails A; \; \cM \mapsto  \pi \uH^0(\cM)_{\geq 0}$
induces an isomorphism of algebraic triples $(\cC,\cO, s) \cong (\tails A,\cA, (1))$.
\end{enumerate}
\end{enumerate}
\end{theorem}

\begin{proof}
(1) First we check that $\tails A$ is $\Hom$-finite.
It is enough to show that $\Hom_\cA(\cA, \cM)=\uH^0(\cM)_0$ is finite dimensional over $k$ for any $M \in \grmod A$.
The condition (b) says that $\uH^0(\cM)_{\geq 0}$ is graded right coherent,
so we have a surjection $F \to \uH^0(\cM)_{\geq 0}$ in $\grmod A$ where $F$ is a finitely generated graded free right $A$-module.
Since $A$ is locally finite, we have $\dim_k \uH^0(\cM)_0 < \infty$.

To prove that $(\cA, (1))$ satisfies (A1), it is enough to check that there exist positive integers $i_1, \dots, i_p \in \NN^+$ and an epimorphism $\bigoplus_{j=1}^p \cA(-i_j) \to \cA$.
Since $A$ is graded right coherent, we have an exact sequence
\[ \bigoplus_{j=1}^p A(-i_j) \to A \to A/A_{\geq 1} \to 0\]
in $\grmod A$. Since $A/A_{\geq 1} \in \tors A$, this induces a desired epimorphism, so (A1) follows.

We next show that $(\cA, (1))$ satisfies (A2). First, note that, for every $M \in \grmod A$ and every $n\in \ZZ$, 
since $A$ is graded right coherent and $M/M_{\geq n}$ is finite dimensional, $M/M_{\geq n}$ is graded right coherent by Proposition \ref{prop.coh}, so $M_{\geq n}$ is also graded right coherent.
Let $\phi: \cM\to \cN$ be an epimorphism in $\tails A$.
Then there exists a homomorphism $\psi: M' \to N/N'$ in $\grmod A$ such that $M/M', N', \coker \psi \in \tors A$, and $\pi (\psi)= \phi$.
It follows that
$$\xymatrix@R=2pc@C=2pc{ M_{\geq n} \ar[r]^{\cong} &M'_{\geq n} \ar[r]^{\psi_{\geq n}} &N/N'_{\geq n} \ar[r]^{\cong} &N_{\geq n} }$$
are surjective for all $n \gg 0$.  Since $\pi M_{\geq n}\cong \pi M, \pi N_{\geq n}\cong \pi N$, we may assume that there exists an epimorphism $\psi:M\to N$ such that $\pi (\psi)=\phi$ by replacing $M, N$ by $M_{\geq n}, N_{\geq n}$. 

An exact sequence $0 \to A_{\geq i} \to A \to A/A_{\geq i} \to 0$ induces the following exact sequence
$$\begin{CD} 0 \to \uH^0_{\fm}(M)_{\geq n} \to M_{\geq n} @>(\phi_M)_{\geq n}>> \uH^0(\cM)_{\geq n} \to \uH^1_{\fm}(M)_{\geq n} \to 0 \end{CD}$$
of graded right $A$-modules (see Example \ref{ex.phi}).  
Since the two middle terms in the above sequence are graded right coherent, we see that $\uH^0_{\fm}(M)_{\geq n}$ and $\uH^1_{\fm}(M)_{\geq n}$ are graded right coherent.
Moreover, since $\uH^0_{\fm}(M)$ and $\uH^1_{\fm}(M)$ are $\fm$-torsion modules, so are $\uH^0_{\fm}(M)_{\geq n}$ and $\uH^1_{\fm}(M)_{\geq n}$.
These imply that $\uH^0_{\fm}(M)_{\geq n}$ and $\uH^1_{\fm}(M)_{\geq n}$ are finite dimensional over $k$.
Hence 
$(\phi _M)_{\geq n}:M_{\geq n} \to \uH^0(\cM)_{\geq n}$
is an isomorphism in $\grmod A$ for every $n \gg 0$.  By applying the same argument for $N$, there exists $m\in \ZZ$ such that both $(\phi _M)_{\geq m}:M_{\geq m} \to \uH^0(\cM)_{\geq m}$ and $(\phi _N)_{\geq m}:N_{\geq m} \to \uH^0(\cN)_{\geq m}$ are isomorphisms in $\grmod A$. 
Since we have the commutative diagram
\[\xymatrix@C=2.5pc@R=2pc{
&\uH^0(M)_{\geq m} \ar@{=}[d]&\uH^0(\cM)_{\geq m}\ar@{=}[d] \\
(\phi_M)_{\geq m}: M_{\geq m} \ar[r]^{\cong} \ar@{->>}@<5ex>[d]_{(\psi)_{\geq m}} &\uHom_A(A, M)_{\geq m} \ar[r]^{\cong} \ar@{->>}[d]^{\uHom_A(A, \psi)_{\geq m}} &\uHom_{\cA}(\cA, \cM)_{\geq m} \ar@{->>}[d]^{\uHom_{\cA}(\cA, \phi)_{\geq m}}\\
(\phi _N)_{\geq m}: N_{\geq m} \ar[r]^{\cong} &\uHom_{A}(A, N)_{\geq m} \ar@{=}[d] \ar[r]^{\cong}&\uHom_{\cA}(\cA, \cN)_{\geq m} \ar@{=}[d]\\
&\uH^0(N)_{\geq m} &\uH^0(\cN)_{\geq m},
}
\]
it follows that 
$$\Hom_{\cA}(\cA(-i), \phi): \Hom_\cA(\cA(-i), \cM) \cong 
\uH^0(\cM)_i \to \uH^0(\cN)_i
\cong \Hom_\cA(\cA(-i), \cN)$$
is surjective for every $i \geq m$.

(2) This follows from \cite[Proposition 2.3, Theorem 2.4]{P} (see also \cite[Theorem 4.5]{AZ}). 
\end{proof}

\begin{definition} Let $A$ be a graded algebra.  A twisting system on $A$ is a sequence $\theta=\{\theta_i\}_{i\in \ZZ}$ of graded $k$-linear automorphisms of $A$ such that $\theta_i(x\theta_j(y))=\theta_i(x)\theta_{i+j}(y)$ for every $i, j\in \ZZ$ and every $x\in A_j, y\in A$.  The twisted graded algebra of $A$ by a twisting system $\theta $ is a graded algebra $A^{\theta}$ where $A^{\theta}=A$ as a graded $k$-vector space with the new multiplication $x*y=x\theta_j(y)$ for $x\in A_j, y\in A$.  
\end{definition} 

If $\s\in \GrAut A$ is a graded algebra automorphism of $A$, then $\{\s^i\}_{i\in \ZZ}$ is a twisting system of $A$.
In this case, we simply write $A^{\s}:=A^{\{\s^i\}}$.  If $B$ is a twisted graded algebra of $A$ by a twisting system, then $\GrMod A\cong \GrMod B$ by \cite {Zh}. 

\begin{lemma} \label{lem.equi}
Let $\cC$ and $\cC'$ be $k$-linear abelian categories.
If $(\cC, \cO, s)$ and $(\cC', \cO', s')$ are equivalent algebraic triples, then the following hold.
\begin{enumerate}
\item{} 
$B(\cC', \cO', s')_{\geq 0}$ is a twisted graded algebra of $B(\cC, \cO, s)_{\geq 0}$ by a twisting system so that $\GrMod B(\cC, \cO, s)_{\geq 0}\cong \GrMod B(\cC', \cO', s')_{\geq 0}$.
\item{}  $(\cO, s)$ is ample for $\cC$ if and only if $(\cO', s')$ is ample for $\cC'$. 
\end{enumerate}
\end{lemma} 

\begin{proof}
Assertion (1) follows from \cite[Theorem 3.3 and Theorem 3.1]{Zh}, and
assertion (2) follows from a straightforward verification.
%
%
\end{proof} 

There is another notion of ampleness introduced in \cite {Min}.  For a ring $R$, a two-sided tilting complex $L$ of $R$ is a complex of $R$-$R$ bimodules such that $-\lotimes _RL$ is an autoequivalence of $\cD(\Mod R)$.  

\begin{definition} \label{def.ttht} 
Let $R$ be a finite dimensional algebra and $L$ a two-sided tilting complex of $R$. 
\begin{enumerate}
\item{} We say that $L$ is quasi-ample if $h^q(L^{\lotimes _Ri})=0$ for all $q\neq 0$ and all $i\geq 0$.  
\item{} We say that $L$ is ample if $L$ is quasi-ample and $(\cD^{L, \geq 0}, \cD^{L, \leq 0})$ is a t-structure on $\cD^b(\mod R)$ where
\begin{align*}
\cD^{L, \geq 0} & :=\{M\in \cD^b(\mod R)\mid h^q(M\lotimes _RL^{\lotimes _Ri})=0 \; \textnormal { for all } q<0, i\gg 0\} \\
\cD^{L, \leq 0} & :=\{M\in \cD^b(\mod R)\mid h^q(M\lotimes _RL^{\lotimes _Ri})=0 \; \textnormal { for all } q>0, i\gg 0\}.
\end{align*}  
The heart of this t-structure is denoted by $\cH^L:=\cD^{L, \geq 0}\cap \cD^{L, \leq 0}$.  
\item{} If $\gldim R=n<\infty$, then the canonical module of $R$ is defined as the two-sided tilting complex $\omega _R:=DR[-n]$.  
\item{} We say that $R$ is (quasi-)Fano if $\omega _R^{-1}:=\RHom_R(\omega _R, R)$ is (quasi-)ample.
\end{enumerate}
\end{definition} 

\begin{remark} The notions of ample and Fano in the above definition were called extremely ample and extremely Fano in \cite {Min}, \cite {MM}, \cite {Mbc}.
\end{remark} 

\subsection{AS-regular Algebras over $R$}

Two generalizations of a notion of AS-regularity were introduced in \cite {MM}. 

\begin{definition}[{\cite[Definition 3.1]{MM}}] 
A locally finite $\NN$-graded algebra $A$ with $A_0 = R$ is called 
AS-regular over $R$ of dimension $d$ and of Gorenstein parameter $\ell$ if the following conditions are satisfied: 
\begin{enumerate}
\item{} $\gldim R<\infty$,
\item{} $\gldim A=d<\infty$, and
\item{} 
$\RuHom_A(R, A) \cong DR(\ell)[-d]$ in $\cD(\GrMod A)$ and in $\cD(\GrMod A^o)$. 
\end{enumerate}

For an AS-regular algebra $A$ over $R$ of Gorenstein parameter $\ell$, we define the Beilinson algebra of $A$ by 
$$\nabla A:=(A^{[\ell]})_0=\begin{pmatrix}
A_0 &A_1 & \cdots & A_{\ell-1} \\
0 & A_0 & \cdots & A_{\ell-2} \\
\vdots & \vdots & \ddots & \vdots \\
0 & 0 & \cdots & A_0\end{pmatrix}.$$
\end{definition}

By \cite[Corollary 3.7]{MM}, a usual AS-regular algebra defined in Definition \ref{def.AS} is exactly an AS-regular algebra over $k$ in the above definition. 
A typical example of an AS-regular algebra over $R$ is given as follows.  For a quasi-Fano algebra $R$ of global dimension $n$, the preprojective algebra of $R$ is defined as the tensor algebra $\Pi R:=T_R(\Ext_R^n(DR, R))$.  

\begin{theorem} 
[{\cite[Corollary 3.12]{Min}, \cite [Theorem 4.2, Theorem 4.12, Theoerm 4.14]{MM}}]
\label{thm.MM} 
If $R$ is a Fano algebra, then 
$\Pi R\cong B(\cD^b(\mod R), R, -\lotimes_R \omega _R^{-1})_{\geq 0}$ is a graded right coherent
AS-regular (Calabi-Yau) algebra of dimension $\gldim R+1$ and of Gorenstein parameter 1 such that $\cD^b(\tails \Pi R)\cong \cD^b(\mod R)$ as triangulated categories.

Conversely, if $A$ is a graded right coherent AS-regular algebra over $R$ of dimension $d\geq 1$, then $\nabla A$ is a Fano algebra of $\gldim \nabla A=d-1$ and $\grmod \Pi \nabla A\cong \grmod A$. 
\end{theorem}

\begin{definition}[{\cite[Definition 3.9]{MM}}] 
A locally finite $\NN$-graded algebra $A$ with $A_0 = R$ is called 
ASF-regular of dimension $d$ and of Gorenstein parameter $\ell$ if the following conditions are satisfied: 
\begin{enumerate}
\item{} $\gldim R<\infty$,
\item{} $\gldim A=d<\infty$, and
\item{} $\RuGa(A) \cong DA(\ell)[-d]$ in $\cD(\GrMod A)$ and in $\cD(\GrMod A^o)$.
\end{enumerate}
\end{definition}

\begin{remark}
In the definition of an ASF-regular algebra given in \cite[Definition 3.9]{MM},
the condition $\gldim R<\infty$ was not imposed.
In this paper, we impose this condition to show that AS-regularity over $R$ and ASF-regularity are equivalent. 
\end{remark} 

In \cite {MM}, Minamoto and the first author showed the following.

\begin{theorem}[{\cite[Theorem 3.12.]{MM}}] \label{thm.ASF}
If $A$ is an AS-regular algebra over $R$ of dimension $d$ and of Gorenstein parameter
$\ell$, then $A$ is an ASF-regular algebra of dimension $d$ and of Gorenstein parameter $\ell$.
\end{theorem}

It was proved that the converse of Theorem \ref{thm.ASF} is also true when $A$ is noetherian (see \cite[Theorem 2.10]{Uct}).
For the purpose of this paper, we here show that a non-noetherian version of the converse of Theorem \ref{thm.ASF}.

\begin{definition}
For a locally finite $\NN$-graded algebra $A$, we say that the condition (EF) holds if
every finite dimensional graded right $A$-module 
is graded right coherent.
\end{definition}

If $A$ is graded right coherent (in particular, right noetherian), then $A$ satisfies (EF) by Proposition \ref{prop.coh}. 
If $A$ is connected graded, then (EF) is equivalent to Ext-finiteness (that is, $\uExt_A^i(k, k)$ is finite dimensional for every $i$).

\begin{lemma}\label{lem.comm}
Let $A$ be a locally finite $\NN$-graded algebra satisfying (EF). Then $\RuGa(-)$ commutes with direct limits.
\end{lemma}

\begin{proof}
The proof is similar to that of \cite[Lemma 4.3]{V} by using (EF) instead of Ext-finiteness.
\end{proof}

Let $A, C$ be graded algebras.  Note that if $M$ is a complex of graded $C$-$A$ bimodules, then $DM$ defined by $(DM)^i=D(M^{-i})$ is a complex of graded $A$-$C$ bimodules.  

\begin{theorem}[Local Duality] \label{thm.ld}
Let $A$ be a locally finite $\NN$-graded algebra, and $C$ another $\NN$-graded algebra.
Assume that $A$ has finite cohomological dimension, and it satisfies (EF). 
Then for any $M \in \cD^b(\GrMod (C^o\otimes_k A))$,
$$D\RuGa(M) \cong \RuHom_A(M, D\RuGa(A))$$
in $\cD(\GrMod (A^o \otimes_k C))$.
\end{theorem}

\begin{proof} 
Using Lemma \ref{lem.comm}, the proof works along the same lines as that of \cite[Theorem 5.1]{V}.
\end{proof}

If $A$ is an ASF-regular algebra, then
there exists a graded algebra automorphism $\nu$ of $A$ such that $D\RuGa(A) \cong A_{\nu}(-\ell)[d]$ in 
$\cD(\GrMod A^e)$, so, similar to the connected graded case, we call the graded algebra automorphism $\nu$ the (generalized) Nakayama automorphism of $A$, and 
we call the graded $A$-$A$ bimodule $\omega_A := A_\nu (-\ell)$ the canonical module over $A$ (see \cite [Section 3.2]{MM}).

\begin{theorem} \label{thm.com}
If $A$ is an ASF-regular algebra of dimension $d$ and of Gorenstein parameter $\ell$ satisfying (EF),
then $A$ is an AS-regular algebra over $R=A_0$ of dimension $d$ and of Gorenstein parameter $\ell$.
\end{theorem}

\begin{proof}
Since $A$ is ASF-regular, we have $D\RuGa(A) \cong A_{\nu}(-\ell)[d]$ in 
$\cD(\GrMod A^e)$.
It follows from Theorem \ref{thm.ld} that
\begin{align*}
\RuHom_A(R,A) &\cong \RuHom_A(R,A_{\nu}(-\ell)[d])_{\nu^{-1}}(\ell)[-d]\\
&\cong \RuHom_A(R,D\RuGa(A))_{\nu^{-1}}(\ell)[-d]\\
&\cong D\RuGa(R)_{\nu^{-1}}(\ell)[-d] \\
&\cong DR_{\nu^{-1}}(\ell)[-d]
\end{align*}
in $\cD(\GrMod A^e)$, so $\RuHom_A(R,A) \cong DR(\ell)[-d]$ in $\cD(\GrMod A)$ and in $\cD(\GrMod A^o)$.
Hence the result follows.
\end{proof}

\begin{remark} Let $A$ be a graded right coherent algebra.  Since $A$ satisfies (EF) by Proposition \ref{prop.coh}, $A$ is an ASF-regular algebra of dimension $d$ and of Gorenstein parameter $\ell$ if and only if $A$ is an AS-regular algebra over $R=A_0$ of dimension $d$ and of Gorenstein parameter $\ell$.  Note that it is conjectured that every AS-regular algebra is graded right coherent. 
\end{remark} 

\section{Regular Tilting Objects and Relative Helices} 

\subsection{Canonical Bimodules}

The canonical sheaf plays an essential role to study a projective scheme in commutative algebraic geometry.  We will define a notion of canonical bimodule for an abelian category.  

\begin{definition}[{\cite[Definition 3.1]{BK}}]
Let $\cC$ be a $\Hom$-finite $k$-linear category.
A Serre functor for $\cC$ is a $k$-linear autoequivalence $S\in \Aut _k\cC$ such that there exists a bifunctorial isomorphism
$$F_{X, Y}:\Hom_{\cC}(X, Y)\to D\Hom_{\cC}(Y, S(X))$$ for $X, Y\in \cC$.
\end{definition}  

\begin{remark} \label{rem.Se}
We explain the functoriality of a Serre functor $S$ in $X$ in the above definition.  Define functors $G=\Hom_{\cC}(-, Y)$ and $H=D\Hom_{\cC}(S^{-1}(Y), -)=\Hom_k(\Hom_{\cC}(S^{-1}(Y), -), k)$.  Fix $\b\in \Hom_{\cC}(X, X')$.  Then 
$$G(\b):\Hom_{\cC}(X', Y)\to \Hom_{\cC}(X, Y)$$
is given by $(G(\b))(\a)=\a\circ \b$.  On the other hand,
$$H(\b):\Hom_k(\Hom_{\cC}(S^{-1}(Y), X'), k)\to \Hom_k(\Hom_{\cC}(S^{-1}(Y), X), k)$$
is given by $((H(\b))(\phi))(\c)=\phi(\b\circ \c)$ for $\c\in \Hom_{\cC}(S^{-1}(Y), X)$.
By functoriality, we have a commutative diagram
$$\begin{CD} 
\Hom_{\cC}(X', Y) @>{G(\b)}>> \Hom_{\cC}(X, Y) \\
@VF_{X', Y}VV @VVF_{X, Y}V \\
D\Hom_{\cC}(S^{-1}(Y), X') @>{H(\b)}>> D\Hom_{\cC}(S^{-1}(Y), X),
\end{CD}$$
so, for $\a\in \Hom_{\cC}(X', Y)$ and $\c\in \Hom_{\cC}(S^{-1}(Y), X)$, we have
$$F_{X, Y}(\a\circ \b)(\c)=(F_{X, Y}(G(\b)(\a)))(\c)=(H(\b)(F_{X', Y}(\a)))(\c)=F_{X', Y}(\a)(\beta\circ \c).$$
\end{remark} 

\begin{definition} Let $\cC$ be an abelian category.  A bimodule $\cM$ over $\cC$ is an adjoint pair of functors from $\cC$ to itself with the suggestive notation $\cM=(-\otimes_{\cC}\cM, \cHom_{\cC}(\cM, -))$.  A bimodule $\cM$ over $\cC$ is invertible if $-\otimes _{\cC}\cM$ is an autoequivalence of $\cC$.  In this case, the inverse bimodule of $\cM$ is defined by $\cM^{-1}=(-\otimes _{\cC}\cM^{-1}, \cHom_{\cC}(\cM^{-1}, -)):=(\cHom_{\cC}(\cM, -), -\otimes _{\cC}\cM)$. 
\end{definition} 

\begin{definition} \label{dfn.cf}
Let $\cC$ be a $k$-linear abelian category.
A canonical bimodule for $\cC$ is 
an invertible bimodule $\omega _{\cC}$ over $\cC$ such that, for some $n\in \ZZ$,
the autoequivalence $-\lotimes _{\cC}\omega _{\cC}[n]$ of $\cD^b(\cC)$ induced by $-\otimes _{\cC}\omega _{\cC}$ is a Serre functor for $\cD^b(\cC)$. 
\end{definition} 

\begin{remark} \label{rem.cms} Let $\cC$ be a $k$-linear abelian category.  
\begin{enumerate}
\item{} Since the Serre functor for $\cD^b(\cC)$ is unique, a canonical bimodule for $\cC$ is unique if it exists.  
\item{} If $\cC$ has a canonical bimodule, then $\cD^b(\cC)$ has a Serre functor by definition, so $\cD^b(\cC)$ is automatically $\Hom$-finite. 
\item{} If $\cC$ has a canonical bimodule $\omega _{\cC}$, and $-\lotimes _{\cC}\omega _{\cC}[n]$ is the Serre functor for $\cD^b(\cC)$, then it is easy to see that $\gldim \cC=n<\infty$. 
\end{enumerate}
\end{remark} 

\begin{example} \label{ex.cf} 
\begin{enumerate}
\item{} If $X$ is a smooth projective scheme, then the canonical sheaf $\omega _X$ over $X$ is the canonical bimodule for $\coh X$. 
\item{} Let $A$ be a noetherian AS-Gorenstein algebra over $k$, and $\omega _A$ the canonical module of $A$.
Then $A$ is a graded isolated singularity 
if and only if $\omega _{\cA}:=\pi \omega _A$ is the canonical bimodule for $\tails A$ (\cite [Theorem 1.3]{U}). 
\item{} If $A$ is a graded right coherent AS-regular algebra over $R$, then $\omega _{\cA}:=\pi \omega_A$ is the canonical bimodule for $\tails A$ where $\omega _A$ is the canonical module over $A$ (\cite [Theorem 4.12]{MM}). 
\item{} If $R$ is a finite dimensional algebra of $\gldim R=n<\infty$, then $-\lotimes _R\omega _R[n]$ is the Serre functor for $\cD^b(\mod R)$, but $\omega _R$ is not a canonical bimodule for $\mod R$ in our sense because $-\otimes _R\omega _R$ is not an autoequivalence of $\mod R$.
However, if $R$ is Fano, then $-\otimes _R\omega _R$ is an autoequivalence of $\cH^{\omega _R^{-1}}$ (see Definition \ref{def.ttht} (2)) and $\cD^b(\cH^{\omega _R^{-1}})=\cD^b(\mod R)$ (see \cite [Corollary 3.6, Corollary 3.12]{Min}), so $\omega _R$ is a canonical bimodule for $\cH^{\omega _R^{-1}}$.  
\end{enumerate}
\end{example} 

\subsection{Regular Tilting Objects}

Let $\cT$ be a triangulated category.
For a set of objects $\{E_0, \dots, E_{r-1}\}$ in $\cT$, we denote by $\<E_0, \dots, E_{r-1}\>$ the smallest full triangulated subcategory of $\cT$ containing $E_0, \dots, E_{r-1}$ closed under isomorphisms and direct summands.

\begin{definition}
Let $\cT$ be a triangulated category.
An object $T\in \cT$ is called tilting if 
\begin{enumerate}
\item{} $\cT=\<T\>$, and 
\item{} $\Hom_{\cT}(T, T[q])=0$ for all $q\neq 0$.
\end{enumerate}
\end{definition}

\begin{remark} \label{rem.Ke} 
If $\cC$ is a $k$-linear abelian category such that $\cD^b(\cC)$ is $\Hom$-finite,
then it is known that $\cD^b(\cC)$ is an algebraic triangulated category (see \cite[Section 1.2 and Section 3.1]{CS})
and Krull-Schmidt (see \cite[Corollary A.2]{CYZ} and \cite[Corollary 2.10]{BS}).
Hence, if $T$ is a tilting object for $\cD^b(\cC)$ such that $\gldim \End_{\cC}(T) <\infty$,
then the functor
\[ \RHom_{\cC}(T, -): \cD^b(\cC) \to \cD^b(\mod \End _{\cC}(T))\]
gives an equivalence of triangulated categories by \cite [Theorem 2.2]{IT}.
\end{remark} 

\begin{definition} \label{def.fano} 
Let $\cC$ be a $k$-linear abelian category having the canonical bimodule $\omega_{\cC}$.
We say that an object $T\in \cD^b(\cC)$ is regular tilting if  
\begin{enumerate} 
\item[(RT1)] $\gldim \End _{\cC}(T)<\infty$,
\item[(RT2)] $\cD^b(\cC)=\<T\>$, and
\item[(RT3)] $\Hom_{\cC}(T, T\lotimes _{\cC}(\omega _{\cC}^{-1})^{\lotimes _{\cC}i}[q])=0$ for all $q\neq 0$ and all $i\geq 0$.
\end{enumerate}
\end{definition} 

The following theorem can be derived from \cite[Theorem 7]{BH}.
For the convenience of the reader, we include our own proof.  

\begin{theorem} \label{thm.rtilt1} 
Let $\cC$ be a $k$-linear abelian category with the canonical bimodule $\omega _{\cC}$, and $T\in \cD^b(\cC)$ a tilting object.
Then $T$ is regular tilting if and only if $R:=\End _{\cC}(T)$ is a quasi-Fano algebra of $\gldim R= \gldim \cC$.
\end{theorem} 

\begin{proof}  Note that since $\cC$ is assumed to have a canonical bimodule, $\cD^b(\cC)$ is $\Hom$-finite.  
$(\Rightarrow)$
Assume that $T$ is a regular tilting object of $\cD^b(\cC)$.
Let $-\lotimes_{\cC}\omega _{\cC}[m]$ be the Serre functor for $\cD^b(\cC)$ and let $\gldim R=n$.
Then we have $m= \gldim \cC$ by Remark \ref{rem.cms} (3).
Using Remark \ref{rem.Ke} and the uniqueness of the Serre functor, we have the following commutative diagram
$$\begin{CD}
\cD^b(\cC) @>\RHom_{\cC}(T, -)>\cong > \cD^b(\mod R) \\
@V-\lotimes_{\cC}\omega _{\cC}[m]V\cong V @V\cong V{-\lotimes _RDR=-\lotimes _R\omega _R[n]}V \\
\cD^b(\cC) @>\cong >\RHom_{\cC}(T, -)> \cD^b(\mod R).
\end{CD}$$
This induces the following commutative diagram 
$$\begin{CD}
\cD^b(\cC) @>\RHom_{\cC}(T, -)>\cong >  \cD^b(\mod R) \\
@V-\lotimes_{\cC}\omega _{\cC}^{-1}V\cong V @V\cong V{-\lotimes _RL}V \\
\cD^b(\cC) @>\cong >\RHom_{\cC}(T, -)> \cD^b(\mod R)
\end{CD}$$
where $L=\omega _R^{-1}[m-n]$. Since 
$$h^q(L^{\lotimes _Ri})\cong \Ext^q_R(R, R\lotimes _R(L^{\lotimes _Ri})) \cong \Ext^q_{\cC}(T, T\lotimes _{\cC}(\omega _{\cC}^{-1})^{\lotimes _{\cC}i}) =0$$
for all $q\neq 0$ and all $i\geq 0$, 
we see that $L$ is a quasi-ample two-sided tilting complex of $R$.
Since $L^{-1}=\omega_R[n-m]=DR[-m]$, it follows that $R$ is a quasi-Fano algebra of $\gldim R=m=\gldim \cC$ by \cite [Remark 1.3]{MM} (cf. \cite [Remark 4.4]{Min}).

$(\Leftarrow)$
If $T \in \cC$ is a tilting object for $\cD^b(\cC)$ and $R=\End _{\cC}(T)$ is a quasi-Fano algebra of $\gldim R=\gldim \cC$,
then 
we have the following commutative diagram
\begin{align} \label{cd:1}
\begin{CD}
\cD^b(\cC) @>\RHom_{\cC}(T, -)>\cong > \cD^b(\mod R) \\
@V-\lotimes_{\cC}\omega _{\cC}^{-1}V\cong V @V\cong V{-\lotimes _R\omega _R^{-1}}V \\
\cD^b(\cC) @>\cong >\RHom_{\cC}(T, -)> \cD^b(\mod R)
\end{CD}
\end{align}
by Remark \ref{rem.Ke}, Remark \ref{rem.cms} (3), and the uniqueness of the Serre functor.
Since $\omega _{R}^{-1}$ is quasi-ample, we have
\[ \Hom_{\cC}(T, T\lotimes _{\cC}(\omega _{\cC}^{-1})^{\lotimes _{\cC}i}[q]) \cong  \Hom_{R}(R, R\lotimes _{R}(\omega _{R}^{-1})^{\lotimes _{R}i}[q])\cong h^q((\omega _{R}^{-1})^{\lotimes _{R}i})=0\]
for all $q\neq 0$ and all $i\geq 0$, so $T$ is a regular tilting object of $\cD^b(\cC)$.
\end{proof}

\begin{theorem} \label{thm.rtilt2} 
Let $\cC$ be a $k$-linear abelian category with the canonical bimodule $\omega _{\cC}$. 
If $T\in \cC$ is regular tilting for $\cD^b(\cC)$ and $(T, -\otimes _{\cC}\omega _{\cC}^{-1})$ is ample for $\cC$, then 
\begin{enumerate}
\item{} $R:=\End _{\cC}(T)$ is a Fano algebra of $\gldim R=\gldim \cC$, and 
\item{} $A:=B(\cC, T, -\otimes _{\cC}\omega _{\cC}^{-1})_{\geq 0} \cong \Pi R$ is a graded right coherent AS-regular (Calabi-Yau) algebra over $R:=\End _{\cC}(T)$ of dimension $\gldim \cC+1$ and of Gorenstein parameter 1. 
\end{enumerate}
\end{theorem} 

\begin{proof}
Since $T$ is regular tilting for $\cD^b(\cC)$, $R$ is a quasi-Fano algebra of $\gldim R=\gldim \cC$  by Theorem \ref{thm.rtilt1}, so
$\omega _R^{-1}$ is a quasi-ample two-sided tilting complex of $R$.
The commutative diagram (\ref{cd:1}) induces the following isomorphisms of graded algebras 
\begin{align*}
A=B(\cC, T , -\otimes _{\cC}\omega_{\cC}^{-1})_{\geq 0} \cong 
B(\cD^b(\cC), T, -\lotimes_{\cC}\omega _{\cC}^{-1})_{\geq 0}\cong B(\cD^b(\mod R), R, - \lotimes _{R}\omega _R^{-1})_{\geq 0}.
\end{align*}
Since $(T, -\otimes _{\cC}\omega _{\cC}^{-1})$ is ample for $\cC$,
$B(\cD^b(\mod R), R, -\lotimes _R\omega _R^{-1})_{\geq 0} \cong A$ 
is a graded right coherent algebra by Theorem \ref{thm.AZ}, so $\omega _R^{-1}$ is an ample two-sided tilting complex of $R$ by \cite [Theorem 3.7]{Min}, hence $R$ is a Fano algebra of $\gldim R=\gldim \cC$.
By Theorem \ref{thm.MM}, 
\begin{align*}
A
\cong B(\cD^b(\mod R), R, - \lotimes _{R}\omega _R^{-1})_{\geq 0}\cong  \Pi R
\end{align*}
is a graded right coherent AS-regular (Calabi-Yau) algebra over $R$ of dimension $\gldim R+1=\gldim \cC+1$ and of Gorenstein parameter 1.   
\end{proof}

\subsection{Relative Helices} 

In this subsection, we will define a ``relaxed" version of a helix.  

\begin{definition} \label{dfn.res}
Let $\cT$ be a $k$-linear triangulated category.  
\begin{enumerate}
\item{} A sequence of objects $\{E_0, \dots, E_{\ell-1}\}$ in $\cT$ is called an exceptional sequence (resp. a relative exceptional sequence) if 
\begin{enumerate}
\item[(RE1)] $\End_{\cT}(E_i)=k$ (resp. $\gldim \End_{\cT}E_i<\infty$) for every $i=0, \dots, \ell-1$, 
\item[(RE2)] $\Hom_{\cT}(E_i, E_i[q])=0$ for every $q\neq 0$ and every $i=0, \dots, \ell-1$, and 
\item[(RE3)] $\Hom_{\cT}(E_i, E_j[q])=0$ for every $q$ and every $0\leq j<i\leq \ell-1$.
\end{enumerate}
A (relative) exceptional sequence $\{E, F\}$ consisting of two objects is called a (relative) exceptional pair.  
\item{} A sequence of objects $\{E_0, \dots, E_{\ell-1}\}$ in $\cT$ is called full if $\<E_0, \dots, E_{\ell-1}\>=\cT$.  
\end{enumerate}
\end{definition} 

\begin{remark} \label{rem.gd}
If $\{E_0, \dots, E_{\ell-1}\}$ is a relative exceptional sequence for a $k$-linear triangulated category $\cT$, then $\gldim \End _{\cT}(\bigoplus _{i=0}^{\ell-1}E_i)<\infty$.  
\end{remark} 

\begin{definition} \label{dfn.rh}
Let $\cC$ be a $k$-linear abelian category having the canonical bimodule $\omega _{\cC}$.
\begin{enumerate}
\item{} A sequence of objects $\{E_i\}_{i\in \ZZ}$ in $\cD^b(\cC)$ is called a (relative) helix of period $\ell$ if, for each $i\in \ZZ$, 
\begin{enumerate}
\item[(H1)] $\{E_{i}, \dots, E_{i+\ell-1}\}$ is a (relative) exceptional sequence for $\cD^b(\cC)$, and 
\item[(H2)] $E_{i+\ell}\cong E_i\lotimes _{\cC}\omega _{\cC}^{-1}$. 
\end{enumerate}
\item{} A relative helix $\{E_i\}_{i\in \ZZ}$ of period $\ell$ is called full if, for each $i\in \ZZ$, $\<E_{i}, \dots, E_{i+\ell-1}\>=\cD^b(\cC)$.  
\item{} A relative helix $\{E_i\}_{i\in \ZZ}$ of period $\ell$ is called geometric if 
$\Hom_{\cT}(E_i, E_j[q])=0$
for every $q\neq 0$ and every $i\leq j$.
\end{enumerate}
\end{definition} 

\begin{definition}
Let $\cT$ be a $k$-linear triangulated category.
For a pair of objects $\{E, F\}$ in $\cT$, we define $\Hom^{\bullet}_{\cT}(E,F) \in \cD(\Mod k)$ by
$(\Hom^{\bullet}_{\cT}(E,F))^i = \Hom_{\cT}(E,F[i])[-i]$
with trivial differentials.  Moreover we define objects $L_E F$ and $R_F E$ in $\cT$ by using distinguished triangles
\begin{align*}
&L_EF \to \Hom^{\bullet}_{\cT}(E,F) \otimes_k E \to F \to, \\
&E \to D\Hom^{\bullet}_{\cT}(E,F) \otimes_k F \to R_FE \to.
\end{align*}
We call $L_E F$ (resp. $R_F E$)  the left mutation of $F$ by $E$ (resp. the right mutation of $E$ by $F$).
\end{definition}

It is known that if $\{E, F\}$ is an exceptional pair,
then $\{L_E F, E\}$ and $\{F, R_FE\}$ are both exceptional pairs, and $R_EL_EF\cong F, L_FR_FE\cong E$.
Mutations of exceptional pairs can be extended to mutations of exceptional sequences.
For a sequence of objects $\e = \{E_0, \dots, E_{\ell-1}\}$,
we define
\begin{align*}
&L_i \e = \{E_0, \dots, E_{i-1}, L_{E_i}E_{i+1}, E_i, E_{i+2}, \dots E_{\ell-1} \},\\
&R_i \e = \{E_0, \dots, E_{i-1}, E_{i+1}, R_{E_{i+1}}E_i, E_{i+2}, \dots E_{\ell-1} \}
\end{align*}
for each $i=0, \dots, \ell-2$. 

\begin{lemma}  [{\cite[Assertion 2.1, Lemma 2.2, Assertion 2.3.a]{B}}] \label{lem.mue}
Let $\e = \{E_0, \dots, E_{\ell-1}\}$ be a sequence of objects in a $k$-linear triangulated category.  For each $i=0, \dots, \ell-2$, the following are equivalent: 
\begin{enumerate}
\item{} $\e$ is a (full) exceptional sequence. 
\item{} $L_i\e$ is a (full) exceptional sequence. 
\item{} $R_i\e$ is a (full) exceptional sequence. 
\end{enumerate}
\end{lemma} 

We inductively define $L^iE_j:=L_{E_{j-i}}(L^{i-1}E_j)$ and $R^iE_j=R_{E_{i+j}}(R^{i-1}E_j)$ for $i\geq 1$.  

\begin{remark} \label{rem.lg1} 
Let $\cC$ be a $k$-linear abelian category.  There is another definition of a helix, which requires the condition 
\begin{enumerate}
\item[(H2)'] $E_{i+\ell}\cong R^{\ell-1}E_i$ (or equivalently, $E_{i-\ell}\cong L^{\ell-1}E_i$)  
\end{enumerate}
in place of (H2) (see \cite [Definition 4.3]{Mbc}). 
If $\cC$ has the canonical bimodule $\omega _{\cC}$,
then $L^{\ell-1}E_i\cong 
E_i\lotimes _{\cC}\omega _{\cC}[\gldim \cC+1-\ell]$ by \cite [Assertion 4.2]{B} and Remark \ref{rem.cms} (3), so the above definition of a helix agrees with the one given in \cite [Definition 4.3]{Mbc} if and only if $\ell=\gldim \cC+1$.
\end{remark} 

\begin{lemma} \label{lem.lr}
Let $\cC$ be a $k$-linear abelian category having the canonical bimodule $\omega _{\cC}$, and $\{E_i\}_{i\in \ZZ}$ a (full) geometric relative helix of period $\ell$ for $\cD^b(\cC)$.   For $r\in \NN^+$ such that $r\mid \ell$, $\{\bigoplus _{i\in I_j}E_i\}_{j\in \ZZ}$ where $I_j=\{i\in \ZZ\mid jr\leq i \leq (j+1)r-1\}$ is a (full) geometric relative helix of period $\ell/r$ for $\cD^b(\cC)$.  In particular, for an algebraic pair $(\cO, s)$ for $\cC$, if $\{s^i\cO\}_{i\in \ZZ}$ is a (full) geometric relative helix of period $\ell$ for $\cD^b(\cC)$, then $\{s^{jr}(\bigoplus _{i=0}^{r-1}s^i\cO)\}_{j\in \ZZ}$ is a (full) geometric relative helix of period $\ell/r$ for $\cD^b(\cC)$.
\end{lemma} 

\begin{proof} First, we show that (H1), that is, $\{\bigoplus _{i\in I_j}E_i, \dots,  \bigoplus _{i\in I_{j+\ell/r-1}}E_i\}$ is a relative exceptional sequence for $\cD^b(\cC)$ for every $j\in \ZZ$.

(RE1) For any $j \in \ZZ$, $\{E_{jr},\dots,E_{(j+1)r-1}\}$ is a relative exceptional sequence, so we have $\gldim \End_{\cC}(\bigoplus_{i\in I_j}E_i)<\infty$ by Remark \ref{rem.gd}.

(RE2) Using the facts that $\{E_{jr},\dots,E_{(j+1)r-1}\}$ is a relative exceptional sequence for any $j \in \ZZ$ and $\{E_i\}_{i\in \ZZ}$ is geometric, we have
$\Hom_{\cC}(\bigoplus _{i\in I_j}E_i, \bigoplus _{i\in I_j}E_i[q])=0$ for every $q\neq 0$ and every $j=0, \dots, \ell/r-1$.

(RE3) For any $i\in \ZZ$ and any $i \leq j_1< j_2 \leq i + \ell/r-1$, if $i_1 \in I_{j_1}$ and $i_2 \in I_{j_2}$, then
$\Hom_{\cC}(E_{i_1}, E_{i_2}[q])=0$ for every $q$ since $0<i_2-i_1\leq \ell-1$,
so we have $\Hom_{\cC}(\bigoplus _{i\in I_{j_1}}E_i, \bigoplus _{i\in I_{j_2}}E_i[q])=0$ for every $q$.

Secondly, since 
$$ \bigoplus _{i\in I_{j+\ell/r}}E_i =\bigoplus _{i\in I_{j}}E_{i+\ell} \cong \bigoplus _{i\in I_j} (E_i\lotimes _{\cC}\omega _{\cC}^{-1}) \cong (\bigoplus _{i\in I_j} E_i)\lotimes _{\cC}\omega _{\cC}^{-1},$$
(H2) is satisfied, so $\{\bigoplus _{i\in I_j}E_i\}_{j\in \ZZ}$ is a relative helix of period $\ell/r$ for $\cD^b(\cC)$.

The full and geometric properties are straightforward.
\end{proof}

\begin{lemma} \label{lem.period1.1} 
Let $\cC$ be a $k$-linear abelian category having the canonical bimodule $\omega _{\cC}$.
If $\{E_i\}_{i\in \ZZ}$ is a full geometric relative helix of period 1 for $\cD^b(\cC)$,
then $E_i$ is a regular tilting object of $\cD^b(\cC)$ for every $i\in \ZZ$. 
\end{lemma}

\begin{proof}
By definition, $\gldim \End_{\cC}E_i<\infty$. Moreover we have
$\Hom_{\cC}(E_i, E_i[q])=0$ for all $q\neq 0$.
Since $\{E_i\}_{i\in \ZZ}$ is a full relative helix of period 1, $\langle E_i \rangle = \cD^b(\cC)$.
These say that $E_i$ is a tilting object of $\cD^b(\cC)$.
Since $\{E_i\}_{i\in \ZZ}$ is a geometric relative helix of period 1, we have
$$\Hom_{\cC}(E_i, E_i\lotimes _{\cC}(\omega _{\cC}^{-1})^{\lotimes _{\cC}j}[q]) =  \Hom_{\cC}(E_i, E_{i+j}[q]) =0$$
for all $q\neq 0$ and all $j\geq 0$, so $E_i$ is regular tilting.
\end{proof}

\begin{lemma} \label{lem.period1.2} 
Let $\cC$ be a $k$-linear abelian category having the canonical bimodule $\omega _{\cC}$.
If $E$ is a regular tilting object of $\cD^b(\cC)$, then $\{E\lotimes _{\cC}(\omega _{\cC}^{-1})^{\lotimes _{\cC}i} \}_{i\in\ZZ}$ is a full geometric relative helix of period 1 for $\cD^b(\cC)$.
\end{lemma}

\begin{proof}
Clearly, $\gldim \End_{\cC}(E\lotimes _{\cC}(\omega _{\cC}^{-1})^{\lotimes _{\cC}i}) = \gldim \End_{\cC}(E) < \infty$. Moreover we have
$$\Hom_{\cC}(E\lotimes _{\cC}(\omega _{\cC}^{-1})^{\lotimes _{\cC}i}, E\lotimes _{\cC}(\omega _{\cC}^{-1})^{\lotimes _{\cC}i}[q])
\cong \Hom_{\cC}(E, E [q])=0$$
for every $q \neq0$. These mean that $\{E\lotimes _{\cC}(\omega _{\cC}^{-1})^{\lotimes _{\cC}i} \}_{i\in\ZZ}$ is a relative helix of period 1 for $\cD^b(\cC)$.
For every $q\neq 0$ and every $i\leq j$,
\[ \Hom_{\cC}(E\lotimes _{\cC}(\omega _{\cC}^{-1})^{\lotimes _{\cC}i}, E\lotimes _{\cC}(\omega _{\cC}^{-1})^{\lotimes _{\cC}j}[q])
\cong \Hom_{\cC}(E, E\lotimes _{\cC}(\omega _{\cC}^{-1})^{\lotimes_{\cC} (j-i)}[q])=0 \]
so $\{E\lotimes _{\cC}(\omega _{\cC}^{-1})^{\lotimes _{\cC}i} \}$ is geometric.
Since $\langle E\lotimes _{\cC}(\omega _{\cC}^{-1})^{\lotimes _{\cC}i} \rangle =\langle E \rangle = \cD^b(\cC)$, it follows that 
$\{E\lotimes _{\cC}(\omega _{\cC}^{-1})^{\lotimes _{\cC}i} \}$ is full.
\end{proof}

\begin{lemma} \label{lem.po} 
Let $(\cC, \cO, s)$ be an algebraic triple. 
For $r\in \NN^+$, $(\cO, s)$ is ample for $\cC$ if and only if $(\bigoplus _{i=0}^{r-1}s^i\cO, s^{r})$ is ample for $\cC$.  
\end{lemma} 

\begin{proof} 
Let $I_j=\{i\in \ZZ\mid jr\leq i\leq (j+1)r-1\}$ so that $(s^r)^j(\bigoplus _{i\in I_0}s^i\cO)\cong \bigoplus _{i\in I_j}s^i\cO$.  
Clearly, (A1) for the pair $(\cO, s)$ is equivalent to (A1) for the pair $(\bigoplus _{i\in I_0}s^i\cO, s^r)$.

For every epimorphism $\phi:\cM\to \cN$ in $\cC$, we see that
$$\Hom_{\cC}(\bigoplus _{i\in I_j}s^i\cO, \phi):\Hom_{\cC}(\bigoplus _{i\in I_j}s^i\cO, \cM)\to \Hom_{\cC}(\bigoplus _{i\in I_j}s^i\cO, \cN)$$
is surjective
if and only if $\Hom_{\cC}(s^i\cO, \phi):\Hom_{\cC}(s^i\cO, \cM)\to \Hom_{\cC}(s^i\cO, \cN)$ is surjective for every $i \in I_j$.
This implies that (A2) for $(\cO, s)$ is equivalent to (A2) for $(\bigoplus _{i\in I_0}s^i\cO, s^r)$.
\end{proof} 

\begin{proposition} \label{prop.am}
Let $\cC$ be a $k$-linear abelian category having the canonical bimodule $\omega _{\cC}$, and $(\cO, s)$ an algebraic pair for $\cC$.  If $\{s^i\cO\}_{i\in \ZZ}$ is a full geometric relative helix of period $\ell$ for $\cD^b(\cC)$, then the following hold. 
\begin{enumerate}
\item{} $T:=\bigoplus _{i=0}^{\ell-1}s^i\cO\in \cC$ is a regular tilting object for $\cD^b(\cC)$.  
\item{} 
$(\cC, T, -\otimes _{\cC}\omega_{\cC}^{-1})\sim (\cC, T, s^{\ell})$.  
\item{} $(\cO, s)$ is ample for $\cC$ if and only if $(T, -\otimes _{\cC}\omega _{\cC}^{-1})$ is ample for $\cC$. 
\end{enumerate}
\end{proposition}

\begin{proof} (1) If $\{s^i\cO\}_{i\in \ZZ}$ is a full geometric relative helix of period $\ell$, then 
$\{s^{j\ell}(\bigoplus _{i=0}^{\ell-1}s^i\cO)\}_{j\in \ZZ}$ is a full geometric relative helix of period 1 by Lemma \ref{lem.lr}, so 
$T:=\bigoplus _{i=0}^{\ell-1}s^i\cO
\in \cC$ is a regular tilting object for $\cD^b(\cC)$ by Lemma \ref{lem.period1.1}. 

(2) Since 
$$(s^{\ell})^jT=(s^{\ell})^j(\bigoplus _{i=0}^{\ell-1}s^i\cO)\cong \bigoplus _{i=0}^{\ell-1}s^{i+j\ell}\cO
\cong (\bigoplus _{i=0}^{\ell-1}s^i\cO) \lotimes _{\cC}(\omega _{\cC}^{-1})^{\lotimes _{\cC}j}= T\lotimes _{\cC}(\omega _{\cC}^{-1})^{\lotimes _{\cC}j}$$ 
for every $j\in \ZZ$, it follows that 
$(\cC, T, -\otimes _{\cC}\omega_{\cC}^{-1})\sim (\cC, T, s^{\ell})$.

(3) By Lemma \ref{lem.equi} (2), $(T, -\otimes _{\cC}\omega_{\cC}^{-1})$ is ample for $\cC$ if and only if $(T, s^{\ell})$ is ample for $\cC$.
By Lemma \ref{lem.po}, $(T, s^{\ell})=(\bigoplus _{i=0}^{\ell-1}s^i\cO, s^{\ell})$ is ample for $\cC$ if and only if $(\cO, s)$ is ample for $\cC$. 
\end{proof}

\section{Main Result} 

We are now ready to state and prove the main result of this paper, which gives a complete answer to Question \ref{ques.intro}.  Note that if $A$ is an AS-regular algebra over $A_0$ of dimension 0, then $A$ is finite dimensional over $k$, so $\tails A$ is trivial.  

\begin{theorem} \label{thm.main}
Let 
$\cC$ be a $k$-linear abelian category.
Then $\cC\cong \tails A$ for some graded right coherent AS-regular algebra over $A_0$ of dimension at least 1 and of Gorenstein parameter $\ell$ if and only if
\begin{enumerate}
\item[(AS1)] $\cC$ has a canonical bimodule $\omega_{\cC}$, and
\item[(AS2)] there exists an ample algebraic pair $(\cO, s)$ for $\cC$ such that $\{s^i\cO\}_{i\in \ZZ}$ is a full geometric relative helix of period $\ell$ for $\cD^b(\cC)$. 
\end{enumerate}

In fact, if (AS1) and (AS2) are satisfied, then $A=B(\cC, \cO, s)_{\geq 0}$ is a graded right coherent AS-regular algebra over $A_0=\End _{\cC}(\cO)$ of dimension $\gldim \cC+1$ and of Gorenstein parameter $\ell$ such that $\cC\cong \tails A$.

In this case, $A$ is right noetherian if and only if $\cO\in \cC$ is a noetherian object. 
\end{theorem}

\begin{proof}  
$(\Rightarrow )$
Let $A$ be a graded right coherent AS-regular algebra over $A_0$ of dimension $d \geq 1$ and of Gorenstein parameter $\ell$. 
Then $\tails A$ has the canonical bimodule $\omega _{\cA}$ by Example \ref{ex.cf}.
By Theorem \ref{thm.MM}, $R:=\nabla A$ is a Fano algebra of $\gldim R=d-1$, and 
$B:=\Pi R$ is a graded right coherent AS-regular algebra over $R$ of dimension $d$ and of Gorenstein parameter 1.  
Since $B$ is a twisted graded algebra of $A^{[\ell]}$ by a graded algebra automorphism 
by \cite [Theorem 4.12]{MM}, there exists an equivalence functor $\grmod A\to \grmod B$ sending
$\bigoplus _{i=0}^{\ell-1}A(i)$ to $B$ by \cite [Remark 4.9]{MM}.
Since we have $\cA(i)\otimes _{\cA}\omega _{\cA}^{-1}\cong \cA(i+\ell)$ for every $i\in \ZZ$, it follows that
\begin{align*}
(\tails A, \bigoplus _{i=0}^{\ell-1}\cA(i), (\ell))
\sim &(\tails A, \bigoplus _{i=0}^{\ell-1}\cA(i), -\otimes _{\cA}\omega_{\cA}^{-1}) \\
\cong &(\tails B, \cB, -\otimes _{\cB}\omega _{\cB}^{-1})\\
\cong &(\cH^{\omega _R^{-1}}, R, -\lotimes _R\omega _R^{-1})
\end{align*} 
by \cite [Corollary 3.12]{Min}.
Since $(R, -\lotimes _{\cC}\omega _R^{-1})$ is ample for $\cH^{\omega _R^{-1}}$ by \cite[Lemma 3.5]{Min}, 
$(\bigoplus _{i=0}^{\ell-1}\cA(i), (\ell))$ is ample for $\tails A$ by Lemma \ref{lem.equi} (2), so $(\cA, (1))$ is ample for $\tails A$ by Lemma \ref{lem.po}.

We next show that $\{\cA(i)\}_{i\in\ZZ}$ is a full geometric relative helix of period $\ell$ for $\cD^b(\tails A)$.
By \cite[Proposition 4.4]{MM}, 
$\End_{\cA}(\cA(i)) \cong \End_{\cA}(\cA) \cong A_0$,
so $\gldim \End_{\cA}(\cA(i)) < \infty$.
Since $\cA(i)\otimes_{\cA}\omega_{\cA}^{-1} \cong \cA(i+\ell)$,
it follows from \cite[Proposition 4.4]{MM} again that $\{\cA(i)\}_{i\in \ZZ}$ is a geometric relative helix of period $\ell$.
Furthermore, similar to the proof of \cite[Proposition 4.3]{MM}, we have 
$\langle \cA(i),\dots, \cA(i+\ell-1) \rangle = \cD^b(\tails A)$ for every $i\in \ZZ$,
so $\{\cA(i)\}_{i\in \ZZ}$ is a full relative helix.

Since $\cC \cong \tails A$, we see that $\cC$ has an ample algebraic pair
$(\cO, s)$ such that $\{s^i\cO \}_{i\in \ZZ}$ is a full geometric relative helix of period $\ell$ for $\cD^b(\cC)$.

$(\Leftarrow )$
Suppose that $\cC$ satisfies (AS1) and (AS2).  Let $n = \gldim \cC$.
Since $(\cO, s)$ is ample for $\cC$, $A:=B(\cC, \cO, s)_{\geq 0}$ is graded right coherent and
$(\cC, \cO, s)\cong (\tails A, \cA, (1))$ by Theorem \ref{thm.AZ}.  
By Proposition \ref{prop.am}, $T:=\bigoplus _{i=0}^{\ell-1}s^i\cO\in \cC$ is a regular tilting object for $\cD^b(\cC)$ and $(T, -\otimes _{\cC}\omega _{\cC}^{-1})$ is ample for $\cC$,
so it follows from Theorem \ref{thm.rtilt2} that $\Pi R\cong  B(\cC, T , -\otimes _{\cC}\omega_{\cC}^{-1})_{\geq 0}$ is a graded right coherent AS-regular algebra over $R:= \End_{\cC}(T)$ of dimension $n +1$.
Moreover, since $\{s^{-\ell+1}\cO, \dots, s^{-1}\cO, \cO\}$ is a relative exceptional sequence,
\begin{align*}
A^{[\ell]} & =(B(\cC, \cO, s)_{\geq 0})^{[\ell]} \cong (B(\cC, \cO, s)^{[\ell]})_{\geq 0}  \cong B(\cC,\bigoplus _{i=0}^{\ell-1}s^i\cO, s^\ell )_{\geq 0}=B(\cC,  T, s^\ell )_{\geq 0}
\end{align*}
by \cite[Lemma 3.8]{Mbc}.
Since $(\cC, T, -\otimes _{\cC}\omega_{\cC}^{-1})\sim (\cC,  T, s^\ell )$ by Proposition \ref{prop.am},
we see that $\GrMod A \cong \GrMod A^{[\ell]} \cong \GrMod \Pi R$ by Lemma \ref{lem.equi} (1),
so we have $\gldim A=\gldim \Pi R=n +1$.  

For the rest, we will show that $A$ is AS-regular over $A_0=\End _{\cC}(\cO)$ of dimension $n+1$ and of Gorenstein parameter $\ell$.  

First assume $n=0$.
It follows from Theorem \ref{thm.rtilt2} that $\gldim R=\gldim \cC=0$, so $R$ is semisimple.
Since $\cD^b(\cC)\cong \cD^b(\mod R)$, we have $-\otimes _{\cC}\omega _{\cC}=\id_{\cC}$, so $s^{j\ell+i}\cO\cong s^i\cO$ for every $i, j\in \ZZ$.
It follows that $(\cC, \cO, s^{\ell})\sim (\cC, \cO, \id_{\cC})$, so $A^{(\ell)}\cong B(\cC, \cO, s^{\ell})_{\geq 0}$ is a twisted graded algebra of $B(\cC, \cO, \id_{\cC})_{\geq 0}\cong A_0[x]$ where $\deg x=1$ by Lemma \ref{lem.equi} (1).  Since $\{\cO, \dots, s^{\ell-1}\cO\}$ is a relative exceptional sequence, $\Hom_{\cC}(\cO, s^i\cO)\cong D\Hom_{\cC}(s^i\cO, \cO)=0$ for every $0<i<\ell$, so 
$$\Hom_{\cC}(\cO, s^i\cO)\cong \begin{cases} \End _{\cC}(\cO)=A_0 & \textnormal { if } i\in \ell \ZZ \\
0 & \textnormal { if } i\not \in \ell\ZZ.\end{cases}$$
It follows that $A=B(\cC, \cO, s)_{\geq 0}$ is a twisted graded algebra of $A_0[x]$ where $A_0$ is semisimple and $\deg x=\ell$, so $A$ is AS-regular over $A_0$ of dimension 1 and of Gorenstein parameter $\ell$.

We now assume $n \geq 1$.
Since $A$ is graded right coherent, it satisfies the condition (EF),
so it is enough to show that $A$ is ASF-regular of dimension $n+1$ and of Gorenstein parameter $\ell$ by Theorem \ref{thm.com}.
Note that we have an exact sequence 
$$\begin{CD} 
0 \to \uH_{\fm}^0(A) \to A @>\phi_A >> \uH^0(\cA) \to \uH_{\fm}^1(A) \to 0 \\
\end{CD}$$
and isomorphisms
$$\uH_{\fm}^q(A) \cong \uH^{q-1}(\cA), \; \;  q\geq 2$$
of graded $A$-$A$ bimodules where 
$\phi_A:A\to \uH^0(\cA)$ is the graded algebra homomorphism defined in Example \ref{ex.phi}. 
Thus it is enough to check that $\phi_A$ above is an isomorphism and 
$$\uH^q(\cA)\cong \begin{cases} 
0 & \textnormal { if } q \neq 0, n \\
DA(\ell) & \textnormal { if } q=n 
\end{cases}$$ 
as graded right and left $A$-modules. 

If $j\geq 0$, then $\Ext^q_{\cC}(\cO, s^j\cO)=0$ for all $q\neq 0$ since $\{s^j\cO\}$ is geometric. 
If $-\ell< j< 0$, then $\Ext^q_{\cC}(\cO, s^j\cO)=0$ since $\{s^{-\ell+1}\cO, \dots, s^{-1}\cO, \cO\}$ is a relative exceptional sequence.  
If $j\leq -\ell$, then $\Ext^q_{\cC}(\cO, s^j\cO)\cong D\Ext^{n-q}_{\cC}(s^j\cO, s^{-\ell}\cO)\cong D\Ext^{n-q}_{\cC}(\cO, s^{-\ell-j}\cO)=0$ for all $q\neq n$ since $\{s^j\cO\}$ is geometric. It follows that $\uH^q(\cA) \cong \bigoplus _{j\in \ZZ}\Ext^q_{\cC}(\cO, s^j\cO)=0$ for all $q\neq 0, n$.  On the other hand, if $-\ell< j< 0$, then $\Hom_{\cC}(\cO, s^j\cO)=0$ since $\{s^{-\ell+1}\cO, \dots, s^{-1}\cO, \cO\}$ is a relative exceptional sequence.  
If $j\leq -\ell$, then $\Hom_{\cC}(\cO, s^j\cO)\cong D\Ext^{n}_{\cC}(s^j\cO, s^{-\ell}\cO)\cong D\Ext^{n}_{\cC}(\cO, s^{-\ell-j}\cO)=0$
since $\{s^j\cO\}$ is geometric.
Thus $A = B(\cC, \cO, s)$.

Recall that the functor $\pi :\grmod A\to\tails A$ induces a morphism of algebraic triples $(\grmod A, A, (1))\to (\tails A, \cA, (1))$ by Example \ref{ex.phi}.  Since $\cC$ is $\Hom$-finite and $(\cO, s)$ is ample for $\cC$, we have a functor $\uH^0(-)_{\geq 0}:\cC\to \grmod A$ by Theorem \ref{thm.AZ} (2).  Since 
$\uH^0(-)_{\geq 0}\circ s=\bigoplus _{i=0}^{\infty}\Hom_{\cC}(\cO, s^{i+1}(-))$ and $(1)\circ \uH^0(-)_{\geq 0}=\bigoplus _{i=-1}^{\infty}\Hom_{\cC}(\cO, s^{i+1}(-))$,
there exists a natural transformation $\uH^0(-)_{\geq 0}\circ s\to (1)\circ \uH^0(-)_{\geq 0}$.  Since $\uH^0(\cO)_{\geq 0}=B(\cC, \cO, s)_{\geq 0}=A$, the functor $\uH^0(-)_{\geq 0}:\cC\to \grmod A$ induces a morphism of algebraic triples $(\cC, \cO, s)\to (\grmod A, A, (1))$.
By Theorem \ref{thm.AZ} (2), the composition of these morphisms is an isomorphism of algebraic triples $(\cC, \cO, s)\to (\tails A, \cA, (1))$.  In the commutative diagram 
$$\begin{CD}
\psi: & B(\cC, \cO, s) @>\uH^0(-)_{\geq 0}>> B(\grmod A, A, (1)) @>\pi >> B(\tails A, \cA, (1)) \\
& \parallel & & \parallel & & \parallel  \\
\phi _A: & A @>>> \uH^0(A) @>>> \uH^0(\cA),
\end{CD}$$
$\psi$ is an isomorphism of a graded algebras by Theorem \ref{thm.AZ} (2), so $\phi_A$ is also an isomorphism of graded algebras.

Consider the diagram 
$$\begin{CD}
\uH^n(\cA)_i\times A_j @>>> \uH^n(\cA)_{i+j} \\
@V\cong VV @VV\cong V \\
\Hom_{\cC}(s^j\cO, s^{i+j}\cO[n])\times \Hom_{\cC}(\cO, s^j\cO) @>\Phi >> \Hom_{\cC}(\cO, s^{i+j}\cO[n]) \\
@VF\times \id VV @VVF V \\
D\Hom_{\cC}(s^{i+j+\ell}\cO, s^j\cO)\times \Hom_{\cC}(\cO, s^j\cO) @>\Psi >> D\Hom_{\cC}(s^{i+j+\ell}\cO, \cO) \\
@V\cong VV @VV\cong V \\
DA(\ell)_i\times A_j @>>> DA(\ell)_{i+j}
\end{CD}$$
where the top and the bottom squares are commutative and $F$ is a map induced by the Serre functor as in Remark \ref{rem.Se}.
For $(\a, \b) \in \Hom_{\cC}(s^j\cO, s^{i+j}\cO[n])\times \Hom_{\cC}(\cO, s^j\cO)$, we have  $\Phi (\a, \b)=\a\circ \b$.
Moreover, for $(\phi, \b) \in D\Hom_{\cC}(s^{i+j+\ell}\cO, s^j\cO)\times \Hom_{\cC}(\cO, s^j\cO)$, we have $\Psi (\phi, \b)(\c)=\phi (\b\circ \c)$ for every $\c\in \Hom_{\cC}(s^{i+j+\ell}\cO, \cO)$.
Since 
$$F(\Phi(\a, \b))(\c)=F(\a\circ \b)(\c)=F(\a)(\b\circ \c)=\Psi (F(\a), \b)(\c)=\Psi((F\times \id)(\a, \b))(\c)$$ 
for every $\c\in \Hom_{\cC}(s^{i+j+\ell}\cO, \cO)$ by Remark \ref{rem.Se}, the above diagram commutes, 
so $\uH^n(\cA)\cong DA(\ell)$ as graded right $A$-modules.
Similarly, we can show that $\uH^n(\cA)\cong DA(\ell)$ as graded left $A$-modules. 
Hence $A$ is ASF-regular of dimension $n +1 \geq 2$ and of Gorenstein parameter $\ell$.

For the last statement, since $\cC$ is $\Hom$-finite, $\H^0(\cM)$ is finite dimensional for every object $\cM\in \cC$.  Since $(\cO, s)$ is ample for $\cC$, if 
$\cO\in \cC$ is a noetherian object, then $A=B(\tails A, \cO, s)_{\geq 0}$ is right noetherian by \cite[Theorem 4.5]{AZ}.  Conversely, since $(\cC, \cO, s)\cong (\tails A, \cA, (1))$, if $A=B(\tails A, \cO, s)_{\geq 0}$ is right noetherian, then $\cA\in \tails A$ is a noetherian object, so $\cO\in \cC$ is a noetherian object. 
\end{proof}

\begin{corollary}
Let $\cC$ be a $k$-linear abelian category.
Then $\cC\cong \tails A$ for some graded right coherent AS-regular algebra over $k$ of dimension at least 1 and of Gorenstein parameter $\ell$ if and only if
\begin{enumerate}
\item[(AS1)] $\cC$ has a canonical bimodule $\omega_{\cC}$, and
\item[(AS2)] there exists an ample algebraic pair $(\cO, s)$ such that $\{s^i\cO\}_{i\in \ZZ}$ is a full geometric helix of period $\ell$ for $\cD^b(\cC)$.
\end{enumerate} 

In fact, if (AS1) and (AS2) are satisfied, then $A=B(\cC, \cO, s)_{\geq 0}$ is a graded right coherent AS-regular algebra over $k$ of dimension $\gldim \cC+1$ and of Gorenstein parameter $\ell$ such that $\cC\cong \tails A$. 

In this case, $A$ is Koszul if and only if $\ell=\gldim \cC+1$ (cf. Remark \ref{rem.lg1}).  
\end{corollary}

\begin{proof} Note that if $A$ is AS-regular over $k$ of dimension $d$ and of Gorenstein parameter $\ell$, and
$\cdots \to F^1 \to F^0 \to k \to 0$ is the minimal free resolution of $k$ over $A$, then $F^d\cong A(-\ell)$, so $k$ has a linear resolution if and only if $\ell=d$.  In the above setting, $d=\gldim \cC+1$, so the last statement holds.  
\end{proof} 

\begin{corollary} \label{cor.main}
Let $\cC$ be a $k$-linear abelian category.
Then $\cC\cong \tails A$ for some graded right coherent AS-regular algebra over $A_0$ of dimension at least 1 if and only if
\begin{enumerate}
\item[(AS1)] $\cC$ has a canonical bimodule $\omega_{\cC}$, and
\item[(AS2)'] there exists a regular tilting object $T\in \cC$ for $\cD^b(\cC)$ such that $(T, -\otimes_{\cC}\omega _{\cC}^{-1})$ is ample for $\cC$. 
\end{enumerate}
\end{corollary}

\begin{proof} If $\cC\cong \tails A$ for some graded right coherent AS-regular algebra over $A_0$ of Gorenstein parameter $\ell$, then $\cC$ has a canonical bimodule $\omega_{\cC}$, and there exists an ample algebraic pair $(\cO, s)$ for $\cC$ such that $\{s^i\cO\}_{i\in \ZZ}$ is a full geometric relative helix of period $\ell$ for $\cD^b(\cC)$ by Theorem \ref{thm.main}.  By Proposition \ref{prop.am}, 
$T:=\bigoplus _{i=0}^{\ell-1}s^i\cO\in \cC$ is a regular tilting object for $\cD^b(\cC)$ and $(T, -\otimes_{\cC}\omega _{\cC}^{-1})$ is ample for $\cC$.  

Conversely, if $T\in \cC$ is a regular tilting object for $\cD^b(\cC)$ such that $(T, -\otimes_{\cC}\omega _{\cC}^{-1})$ is ample for $\cC$, then 
$\{T\otimes _{\cC}(\omega _{\cC}^{-1})^{\otimes _{\cC}i} \}_{i\in\ZZ} = \{T\lotimes _{\cC}(\omega _{\cC}^{-1})^{\lotimes _{\cC}i} \}_{i\in\ZZ}$ is a full geometric relative helix of period 1 for $\cD^b(\cC)$ by Lemma \ref{lem.period1.2}, so the result follows from Theorem \ref{thm.main}.
\end{proof} 

\begin{corollary} \label{cor.qv} 
Let $A$ be a graded right coherent (noetherian) AS-regular algebra over $A_0$ of dimension $d$ and of Gorenstein parameter $\ell$.  For $r\in \NN^+$ such that $r\mid \ell$, $B:=A^{[r]}$ is a graded right coherent (noetherian) AS-regular algebra over $B_0$ of dimension $d$ and of Gorenstein parameter $\ell/r$. 
\end{corollary}

\begin{proof}  Since $(\cA, (1))$ is an ample algebraic pair for $\tails A$ such that $\{\cA(i)\}_{i\in \ZZ}$ is a full geometric relative helix of period $\ell$ for $\cD^b(\tails A)$ by the proof of Theorem \ref{thm.main},  $(\bigoplus _{i=0}^{r-1}\cA(i), (r))$ is an ample algebraic pair for $\tails A$ such that $\{(\bigoplus _{i=0}^{r-1}\cA(i))(rj)\}_{j\in \ZZ}$ is a full geometric relative helix of period $\ell/r$ for $\cD^b(\tails A)$ by Lemma \ref{lem.po} and Lemma \ref{lem.lr}.
Since $B=A^{[r]}\cong B(\cD^b(\tails A), \bigoplus _{i=0}^{r-1}\cA(i), (r))_{\geq 0}$, 
we see that $B$ is a graded right coherent AS-regular algebra over $B_0$ of dimension $\gldim (\tails A)+1=d$ and of Gorenstein parameter $\ell/r$ by Theorem \ref{thm.main}. 
\end{proof} 

\begin{example} \label{ex.qvas} 
In the above Corollary, the condition $r\mid \ell$ cannot be dropped.  
For example, if $A=k[x]$ with $\deg x=3$, then $A$ is AS-regular over $k$ of dimension 1 and of Gorenstein parameter 3.  If $B:=A^{[2]}$ is AS-regular over $B_0$, then $\gldim B=\gldim A=1$, so $B=B_0[x]$ as a graded vector space by the proof of Theorem \ref{thm.main} (see also \cite [Theorem 4.15]{MM}).  Since 
$$B:=A^{[2]}=\begin{pmatrix} k & 0 \\ 0 & k \end{pmatrix}\oplus \begin{pmatrix} 0 & kx \\ 0 & 0 \end{pmatrix}\oplus \begin{pmatrix} 0 & 0 \\ kx & 0 \end{pmatrix}\oplus \begin{pmatrix} kx^2 & 0 \\ 0 & kx^2 \end{pmatrix}\oplus \cdots  $$ 
it is not the case, so $B$ is not AS-regular over $B_0=k\times k$. 
Since $\GrMod A^{[r]}\cong \GrMod A$ for every $r\in \NN^+$, this example shows that AS-regularity is not a graded Morita invariant if we do not require algebras to be connected graded (compare with \cite [Theorem 1.3]{Zh}).
\end{example} 

\section{Smooth Quadric Surfaces in a Quantum $\PP^3$} 

It is well-known that, for a smooth quadric surface $Q$ in $\PP^3$,
there exists a noetherian AS-regular algebra $B=k\<x, y\>/(x^2y-yx^2, xy^2-y^2x)$ of dimension 3 and of Gorenstein parameter 4 such that $\coh Q\cong \tails B$.
In this section, we will prove a noncommutative generalization of this result as an application of the main result of this paper.  

Throughout this section, we assume that $k$ is an algebraically closed field of characteristic $0$.

\begin{definition}[\cite {SV}]
We say that a $k$-linear abelian category $\cC$ is a smooth quadric surface in a quantum $\PP^3$ if $\cC\cong \tails S/(f)$ 
where $S$ is a 4-dimensional noetherian quadratic AS-regular algebra over $k$ and $f\in S_2$ is a central regular element such that $A=S/(f)$ is a domain and a graded isolated singularity. 
\end{definition}

Let $S$ be a $4$-dimensional noetherian quadratic AS-regular algebra over $k$.
Then the Hilbert series of $S$ is $H_S(t)= 1/(1-t)^{4}$, and $S$ is a Koszul domain.
Let $f \in S_2$ be a central regular element and $A=S/(f)$.
Then $A$ is a noetherian AS-Gorenstein Koszul algebra of dimension $3$ and of Gorenstein parameter $2$.
There exists a central regular element $z \in A^!$ of degree $2$ such that $A^!/(z) \cong S^!$ where $A^!, S^!$ are Koszul duals of $A, S$.
We define $C(A) := A^![z^{-1}]_0$.

We call $M\in \grmod A$ graded maximal Cohen-Macaulay if $\depth M = \lcd M = \lcd A \; (=3)$ or $M=0$.
It is well-known that $M\in \grmod A$ is graded maximal Cohen-Macaulay if and only if $\uExt^i_A(M ,A) =0$ for all $i \neq 0$.
We write $\CM^{\ZZ}(A)$ for the full subcategory of $\grmod A$ consisting of graded maximal Cohen-Macaulay modules.

\begin{proposition} \label{prop.C(A)}
Let $S$ be a $4$-dimensional noetherian quadratic AS-regular algebra over $k$, and 
$f \in S_2$ a central regular element.  
If $A=S/(f)$ is a domain, then the following are equivalent:
\begin{enumerate}
\item $A$ is a graded isolated singularity. 
\item $A$ is of finite Cohen-Macaulay representation type (i.e, there exist only finitely many indecomposable graded maximal Cohen-Macaulay modules up to isomorphisms and degree shifts).
\item $C(A)$ is a semisimple ring.
\item $C(A) \cong M_2(k) \times M_2(k)$.
\end{enumerate}
\end{proposition}

\begin{proof}
(1) $\Rightarrow$ (3) follows from \cite[Theorem 5.6]{SV}.
(3) $\Rightarrow$ (2) follows from \cite[Proposition 4.1]{Ucm}.
(2) $\Rightarrow$ (1) follows from \cite[Theorem 3.4]{Ucm}.
(3) $\Leftrightarrow$ (4) follows from \cite[Proposition 5.3]{SV}.
\end{proof}

\emph{For the rest, we assume that $S$ is a $4$-dimensional noetherian quadratic AS-regular algebra over $k$, 
$f \in S_2$ is a central regular element, and 
$A=S/(f)$ is a domain and a graded isolated singularity}.
In this case, $\gldim (\tails A)=2$ and $\tails A$ has the canonical bimodule $\omega _{\cA}$ such that $\cM\otimes _{\cA}\omega _{\cA}=\cM_{\nu}(-2)$ for $\cM\in \tails A$ by Example \ref{ex.cf}
where $\nu$ is the Nakayama automorphism of $A$. We define
\[ {\mathbb M} =\{ M \in \CM^{\ZZ}(A) \mid M \;\textnormal{is indecomposable}, M_0 \cong k^2, M = M_0A \}/_{\cong}.\]
The stable category of graded maximal Cohen-Macaulay modules, denoted by $\uCM^{\ZZ}(A)$,
has the same objects as $\CM^{\ZZ}(A)$ and the morphism set is given by
\[ \Hom_{\uCM^{\ZZ}(A)}(M, N) = \Hom_{A}(M,N)/P(M,N) \]
for any $M,N \in \CM^{\ZZ}(A)$, where $P(M,N)$ consists of degree zero $A$-module homomorphisms that factor through a projective module in $\GrMod A$.
Since $A$ is AS-Gorenstein, $\uCM^{\ZZ}(A)$ is a triangulated category
with respect to the translation functor $M[-1]= \Omega M$ (see \cite{Bu}).

\begin{proposition} \label{prop.SV1}
The following hold.
\begin{enumerate}
\item There exists a duality $G: \uCM^{\ZZ}(A) \xrightarrow{\sim} \cD^b(\mod C(A))$ such that
$G(M)$ is a simple $C(A)$-module for every $M\in \mathbb M$.
\item $\mathbb M$ consists of two non-isomorphic modules, say $X$ and $Y$.
Moreover every non-projective indecomposable graded maximal Cohen-Macaulay module is isomorphic to $X(i)$ or $Y(i)$ for some $i\in \ZZ$.
\end{enumerate}
\end{proposition}

\begin{proof}
(1) This follows from \cite[Proposition 5.2 (1)]{SV} and \cite[Proposition 5.4]{SV}.

(2) By Proposition \ref{prop.C(A)}, $C(A)$ has two non-isomorphic simple modules,
so it follows from \cite[Proposition 5.4]{SV} that $\mathbb M$ consists of two non-isomorphic modules.
The last statement follows from the proof of \cite[Proposition 4.1]{Ucm}.
\end{proof}

In \cite{SV}, Smith and Van den Bergh developed the theory of smooth quadric surfaces in a quantum $\PP^3$.
By \cite[Lemma 5.5, Proposition 5.4 (1)]{SV}, if $M \in \mathbb M$, then $\Omega M(1) \in {\mathbb M}$ and $\Omega^2 M(2) \cong M$,
so, by Proposition \ref{prop.SV1} (2), the following two cases may occur.
\begin{description}
\item[(\bf standard)] $\Omega X(1) \cong Y$ and $\Omega Y(1) \cong X$. 
\item[(\bf non-standard)] $\Omega X(1) \cong X$ and $\Omega Y(1) \cong Y$. 
\end{description}
In \cite[Theorem 5.8]{SV}, Smith and Van den Bergh claim that
if $A=S/(f)$ is a domain and a graded isolated singularity, then $A$ is standard.
However, in the appendix, we will give an example of a non-standard algebra $A=S/(f)$ which is a domain and a graded isolated singularity.

\begin{definition}
We say that a smooth quadric surface $\cC$ in a quantum $\PP^3$ is standard if $\cC\cong \tails A$ where $A=S/(f)$ is standard.
\end{definition}

Note that a smooth quadric surface $Q$ in $\PP^3$ is standard (see the appendix).
To give an application of the main result of this paper, we focus on standard smooth quadric surfaces in a quantum $\PP^3$.  

\emph{For the rest, we assume that $A=S/(f)$ is a domain,  a graded isolated singularity and standard}.

\begin{proposition} \label{prop.SV2}
The following hold.
\begin{enumerate}
\item We have exact sequences 
$$0 \to Y(-1) \to A^2 \to X \to 0 \quad \textrm{and} \quad 0 \to X(-1) \to A^2 \to Y \to 0.$$
\item $\Omega^{2i+1} X \cong Y(-2i-1)$, $\Omega^{2i}X \cong X(-2i)$,  $\Omega^{2i+1} Y \cong X(-2i-1)$, and $\Omega^{2i}Y \cong Y(-2i)$ for every $i\in \ZZ$.
\item $\uCM^{\ZZ}(A) \cong \langle Y, X\rangle$.
\end{enumerate}
\end{proposition}

\begin{proof}
Since $A$ is standard, (1) and (2) follows from \cite[Proposition 5.4 (2)]{SV}.
For (3), by Proposition \ref{prop.SV1}, $\cD^b(\mod C(A))=\<G(X), G(Y)\>$, so $\uCM^{\ZZ}(A) \cong \langle Y, X\rangle$.
\end{proof}

The Auslander-Reiten quiver of $\uCM^{\ZZ}(A)$ is given as follows:
\[\xymatrix@C=1.5pc@R=0.8pc{
\cdots &X(-2) \ar@{.>}[l] &Y(-1) \ar@{.>}[l] &X \ar@{.>}[l] &Y(1) \ar@{.>}[l] &X(2) \ar@{.>}[l] &\cdots  \ar@{.>}[l]\\
\cdots &Y(-2) \ar@{.>}[l] &X(-1) \ar@{.>}[l] &Y \ar@{.>}[l] &X(1) \ar@{.>}[l] &Y(2) \ar@{.>}[l] &\cdots  \ar@{.>}[l]
}\]
where dotted arrows show the Auslander-Reiten translation $\t$ in $\uCM^{\ZZ}(A)$. (There are no arrows.)

\begin{lemma} \label{lem.Hil}
The following hold.
\begin{enumerate}
\item $H_A(t) = (1+t)(1-t)^{-3}$.
\item $H_X(t) = 2(1-t)^{-3}$.
\item $H_{\uHom_A(X,A)}(t) = 2t(1-t)^{-3}$. 
\item $H_{\uHom_A(X,X)}(t) = (1+t)(1-t)^{-3}$. 
\item $H_{\uHom_A(X,Y)}(t) = t(3-t)(1-t)^{-3}$. 
\item $H_{\uExt^1_A(X,A)}(t) = 0$. 
\item $H_{\uExt^1_A(X,X)}(t) = 0$. 
\item $H_{\uExt^1_A(X,Y)}(t) = t^{-1}$. 
\end{enumerate}
(These are also true if we exchange the roles of $X$ and $Y$.)
\end{lemma}

\begin{proof}
(1) Since $H_S(t) = (1-t)^{-4}$, the result follows from the exact sequence $0 \to S(-2) \to S \to A \to 0$.

(2) This is \cite[Proposition 5.4 (4)]{SV}.

(3) We have an exact sequence $0 \to X(-2) \to A(-1)^2 \to A^2 \to X \to 0$ by Proposition \ref{prop.SV2} (1). Since $X \in \CM^{\ZZ}(A)$, we have an exact sequence
$0 \to \uHom_A(X,A) \to A^2\to A(1)^2 \to \uHom_A(X,A)(2) \to 0$, so the result follows.
  
(6) Since $X \in \CM^{\ZZ}(A)$, this is clear.

(7) We have 
\begin{align*}
\uExt^1_A(X,X)_i &\cong \Ext^1_{\GrMod A}(X,X(i)) \\
&\cong \Hom_{\uCM^{\ZZ}(A)}(X,X(i)[1])\\ 
&\cong \Hom_{\uCM^{\ZZ}(A)}(\Omega X,X(i))\\
&\cong \Hom_{\uCM^{\ZZ}(A)}(Y(-1),X(i)).
\end{align*}
By the structure of the Auslander-Reiten quiver of $\uCM^{\ZZ}(A)$, this is zero for any $i$. 

(8) We have 
\begin{align*}
\uExt^1_A(X,Y)_i &\cong \Ext^1_{\GrMod A}(X(1),Y(i+1)) \\
&\cong \Hom_{\uCM^{\ZZ}(A)}(X(1),Y(i+1)[1])\\
&\cong \Hom_{\uCM^{\ZZ}(A)}(\Omega X(1),Y(i+1))\\
&\cong \Hom_{\uCM^{\ZZ}(A)}(Y,Y(i+1)).
\end{align*}
If $i=-1$, then 
$\dim_k \Hom_{\uCM^{\ZZ}(A)}(Y,Y) = \dim_k \Hom_{\cD^b(\mod C(A))}(G(Y),G(Y)) = 1$
because $C(A)$ is semisimple and $G(Y)$ is a simple $C(A)$-module by Proposition \ref{prop.SV1} (1).
If $i \neq -1$, then it follows from the structure of the Auslander-Reiten quiver of $\uCM^{\ZZ}(A)$ that $\Hom_{\uCM^{\ZZ}(A)}(Y(-1),Y(i))=0$.
Thus we get the result.

(4) Since we have exact sequences $0 \to Y(-1) \to A^2 \to X \to 0$ and $0 \to X(-2) \to A(-1)^2 \to Y(-1) \to 0$ by Proposition \ref{prop.SV2} (1),
we obtain exact sequences
{\small
\begin{align*}
&0 \to \uHom_A(X,X) \to \uHom_A(A,X)^2 \to \uHom_A(Y(-1),X) \to \uExt^1_A(X,X) =0 \;\; (\textnormal{by (7)}),  \\
&0 \to \uHom_A(Y(-1),X) \to \uHom_A(A(-1),X)^2 \to \uHom_A(X(-2),X) \to \uExt^1_A(Y(-1),X) \to 0.
\end{align*}
}
Combining these, we get 
$$H_{\uHom_A(X,X)}(t) - 2H_{X}(t) +2t^{-1}H_{X}(t) - t^{-2}H_{\uHom_A(X,X)}(t) + t^{-1}H_{\uExt^1_A(Y,X)}(t) = 0,$$
so the result follows from (2), (8).

(5) Since we have an exact sequence $0 \to Y(-1) \to A^2 \to X \to 0$, we obtain an exact sequence
$0 \to \uHom_A(X,Y) \to \uHom_A(A,Y)^2 \to \uHom_A(Y(-1),Y) \to \uExt^1_A(X,Y) \to 0$, so it follows that
\[ H_{\uHom_A(X,Y)}(t) - 2H_{Y}(t) +t^{-1}H_{\uHom_A(Y,Y)}(t) - H_{\uExt^1_A(X,Y)}(t) = 0. \]
The result follows from (2), (4), (8).
\end{proof}

The graded singularity category of $A$ is defined by the Verdier localization 
$$\cD_{\rm Sg}^{\rm gr}(A) := \cD^b(\grmod A)/ \cD^b(\grproj A)$$
where $\grproj A$ is the full subcategory of $\grmod A$ consisting of projective modules.
We denote the localization functor by $\upsilon: \cD^b(\grmod A)\to \cD_{\rm Sg}^{\rm gr}(A)$.
Moreover there exists the equivalence $\uCM^{\ZZ}(A) \xrightarrow{\sim} \cD_{\rm Sg}^{\rm gr}(A)$ by Buchweitz \cite{Bu}.
Since the Gorenstein parameter of $A$ is $\ell =2 >0$, there exists the embedding
$\Phi:=\Phi_0:\cD_{\rm Sg}^{\rm gr}(A)\to \cD^b(\tails A)$ by Orlov \cite{O}.

\begin{lemma}[{\cite[Lemma 4.1]{MU}}] \label{lem.Ami}
Let $M\in \grmod A$.
If $M=M_{\geq 0}$ and $\Hom_{A}(M, A(i))=0$ for all $i\leq 0$, then $\Phi (\upsilon M)\cong  \pi M = \cM$. 
\end{lemma} 

\begin{lemma} \label{lem.AmiXY}
We have $\Phi\cD_{\rm Sg}^{\rm gr}(A) = \langle \cY, \cX \rangle$.
\end{lemma} 

\begin{proof}
By Lemma \ref{lem.Hil} (2), (3) and Lemma \ref{lem.Ami}, it follows that $\Phi (\upsilon X)\cong \cX$ and $\Phi (\upsilon Y)\cong \cY$.
Under the equivalence $\uCM^{\ZZ}(A) \xrightarrow{\sim} \cD_{\rm Sg}^{\rm gr}(A)$, $X, Y$ correspond to $\upsilon X, \upsilon Y$,
so $\cD_{\rm Sg}^{\rm gr}(A) = \langle \upsilon Y, \upsilon X \rangle$ by Proposition \ref{prop.SV2} (3).
Hence $\Phi\cD_{\rm Sg}^{\rm gr}(A) = \langle \Phi (\upsilon Y), \Phi (\upsilon X) \rangle = \langle \cY, \cX \rangle$.
\end{proof}

\begin{lemma} \label{lem.HE}
For $M,N \in \CM^{\ZZ}(A)$, the following hold. 
\begin{enumerate}
\item $\uHom_\cA(\cM, \cN) \cong \uHom_A(M, N)$.
\item $\uExt^1_\cA(\cM, \cN) \cong \uExt^1_A(M, N)$.
\item $\uExt^2_\cA(\cM, \cN) \cong D\uHom_A(N, M_\nu(-2))$.
\item $\uExt^q_\cA(\cM, \cN) =0$ for $q \geq 3$.
\end{enumerate}
\end{lemma}

\begin{proof}
Since $M/M_{\geq n}$ is  a finite dimensional module and $\depth N =3$, we have
\[\lim_{n \to \infty} \uHom_A(M/M_{\geq n}, N) = \lim_{n \to \infty}  \uExt^1_A(M/M_{\geq n}, N) = \lim_{n \to \infty}  \uExt^2_A(M/M_{\geq n}, N) =0, \]
so it follows from \cite[Proposition 7.2 (1)]{AZ} that 
\begin{align*}
&\uHom_\cA(\cM, \cN) \cong \lim_{n \to \infty} \uHom_A(M_{\geq n}, N) \cong \uHom_A(M, N) , \\
&\uExt^1_\cA(\cM, \cN) \cong \lim_{n \to \infty} \uExt^1_A(M_{\geq n}, N) \cong \uExt^1_A(M, N).
\end{align*}
This proves (1), (2). Moreover, since $\depth M_{\nu}(-2) =\depth M=3$,
\begin{align*}
\uExt^2_\cA(\cM, \cN) &\cong D\uHom_\cA(\cN, \cM\otimes_{\cA} \omega _{\cA})\\
&\cong D\uHom_\cA(\cN, \cM_\nu(-2))\\
&\cong D\uHom_A(N, M_\nu(-2))
\end{align*}
by the Serre duality and (1). This shows (3). Since $\gldim (\tails A)=2$, (4) holds.
\end{proof}

\begin{lemma} \label{lem.Hom}
The following hold.
\begin{enumerate}
\item $\Hom_\cA(\cA(i), \cA(j))=
\begin{cases} k & \; \textrm{if}\; j-i=0 \\ 0 & \; \textrm{if}\; j-i\leq -1.\end{cases}$
\item $\Hom_\cA(\cA(i), \cX(j))=
\begin{cases} k^2 & \textrm{if}\; j-i=0 \\ 0 & \textrm{if}\; j-i\leq -1. \end{cases}$
\item $\Hom_\cA(\cX(i), \cA(j))=0$ \hspace{.27in} if $j-i\leq 0$.
\item $\Hom_\cA(\cX(i), \cX(j))=
\begin{cases} k & \; \textrm{if}\; j-i=0 \\ 0 & \; \textrm{if}\; j-i\leq -1. \end{cases}$
\item $\Hom_\cA(\cX(i), \cY(j))=0$ \hspace{.27in} if $j-i\leq 0$.
\end{enumerate}
(These are also true if we exchange the roles of $\cX$ and $\cY$.)
\end{lemma}

\begin{proof}
This follows from Lemma \ref{lem.HE} (1) and Lemma \ref{lem.Hil}.
\end{proof}

\begin{lemma}\label{lem.Ext1}
The following hold.
\begin{enumerate}
\item $\Ext^1_\cA(\cA(i), \cA(j))=0$.
\item $\Ext^1_\cA(\cA(i), \cX(j))=0$.
\item $\Ext^1_\cA(\cX(i), \cA(j))=0$.
\item $\Ext^1_\cA(\cX(i), \cX(j))=0$.
\item $\Ext^1_\cA(\cX(i), \cY(j))=0$ if $j-i \neq -1$.
\end{enumerate}
(These are also true if we exchange the roles of $\cX$ and $\cY$.)
\end{lemma}

\begin{proof}
This follows from Lemma \ref{lem.HE} (2) and Lemma \ref{lem.Hil}.
\end{proof}

\begin{lemma} \label{lem.nu}
We have $X_\nu \cong X$ and $Y_\nu \cong Y$.
\end{lemma}

\begin{proof}
Since $C(A)$ is semisimple, $\cD^b(\mod C(A))$ has the Serre functor $\id_{\cD^b(\mod C(A))}$,
so $\uCM^{\ZZ}(A)$ has the Serre functor $\id_{\uCM^{\ZZ}(A)}$ by Proposition \ref{prop.SV1} (1).
However, by \cite[Corollary 4.5]{U}, $\uCM^{\ZZ}(A)$ has the Serre functor $-\otimes _{A}\omega _{A}[2]\cong \Omega ^{-2}(-)_\nu(-2)$. Since $(-)_{\nu}$ commutes with shifts, we have
$X \cong \Omega ^2X(2)\cong \Omega ^{-2}(\Omega ^2X(2))_\nu(-2) \cong X_\nu$
in $\uCM^{\ZZ}(A)$ by Proposition \ref{prop.SV2}. Since $X$ is indecomposable and non-projective, this means that $X \cong X_\nu$ in $\CM^{\ZZ}(A)$.
\end{proof}

\begin{lemma} \label{lem.Ext2}
The following hold.
\begin{enumerate}
\item $\Ext^2_\cA(\cA(i), \cA(j))=0$ if $j-i \geq -1$.
\item $\Ext^2_\cA(\cA(i), \cX(j))=0$ if $j-i \geq -2$.
\item $\Ext^2_\cA(\cX(i), \cA(j))=0$ if $j-i \geq -1$.
\item $\Ext^2_\cA(\cX(i), \cX(j))=0$ if $j-i \geq -1$.
\item $\Ext^2_\cA(\cX(i), \cY(j))=0$ if $j-i \geq -2$.
\end{enumerate}
(These are also true if we exchange the roles of $\cX$ and $\cY$.)
\end{lemma}

\begin{proof}
We only show (5).  The rest are similar.
It follows from Lemma \ref{lem.HE} (3) and Lemma \ref{lem.nu} that
\begin{align*}
\Ext^2_\cA(\cX(i), \cY(j)) &\cong \uExt^2_\cA(\cX(i), \cY(j))_0\\
&\cong D\uHom_A(Y(j), X_\nu(i-2))_0\\
&\cong D(\uHom_A(Y, X)_{i-j-2}),
\end{align*}
so the result follows by Lemma \ref{lem.Hil} (5). 
\end{proof}

\begin{lemma} \label{lem.mueXY}
The left mutation $L_{\cA(i)}\cY(i)$ is $\cX(i-1)$ for any $i$.
\end{lemma}

\begin{proof}
By Proposition \ref{prop.SV2} (1), we have an exact sequence $0 \to X(-1) \to Y_0 \otimes_k A \to Y \to 0$ in $\grmod A$, so we have $0 \to X(i-1) \to Y_0 \otimes_k A(i) \to Y(i) \to 0$ in $\grmod A$. 
This induces an exact sequence
$0 \to \cX(i-1) \to Y_0 \otimes_k \cA(i) \to \cY(i) \to 0 $ in $\tails A$.
By Lemmas  \ref{lem.HE} (4), \ref{lem.Hom}, \ref{lem.Ext1} and \ref{lem.Ext2}, $Y_0 = \Hom_\cA(\cA(i), \cY(i))$ and $\Ext^q_\cA(\cA(i), \cY(i)) = 0$ for $q \neq 0$,
so we obtain a distinguished triangle
$$ \cX(i-1) \to \Hom_{\cA}^\bullet(\cA(i), \cY(i)) \otimes_k \cA(i) \to \cY(i) \to$$
in $\cD^b(\tails A)$.
This means that the left mutation of $\cY(i)$ by $\cA(i)$ is given by $\cX(i-1)$.
\end{proof}

\begin{theorem} \label{thm.fe2}
$\{\cA(-1), \cX(-1), \cA, \cX\}$ is a full exceptional sequence for $\cD^b(\tails A)$.
\end{theorem}

\begin{proof}
By Lemmas \ref{lem.HE} (4), \ref{lem.Hom}, \ref{lem.Ext1}, and \ref{lem.Ext2}, for $i=0, -1$ and any $q \geq 1$, we have
\begin{align*}
&\Hom_\cA(\cA(i), \cA(i)) = k, &&\Hom_\cA(\cX(i), \cX(i)) = k,\\
&\Ext^q_\cA(\cA(i), \cA(i)) = 0, &&\Ext^q_\cA(\cX(i), \cX(i)) = 0,\\
&\Hom_\cA(\cX(-1), \cA(-1)) = 0, &&\Hom_\cA(\cA, \cA(-1)) = 0,  &&\Hom_\cA(\cX, \cA(-1))= 0,\\
&\Hom_\cA(\cA, \cX(-1)) = 0,  &&\Hom_\cA(\cX, \cX(-1)) = 0, &&\Hom_\cA(\cX, \cA) = 0,\\
&\Ext^q_\cA(\cX(-1), \cA(-1)) = 0, &&\Ext^q_\cA(\cA, \cA(-1)) = 0,  &&\Ext^q_\cA(\cX, \cA(-1))= 0,\\
&\Ext^q_\cA(\cA, \cX(-1)) = 0,  &&\Ext^q_\cA(\cX, \cX(-1)) = 0, &&\Ext^q_\cA(\cX, \cA) = 0,
\end{align*}
so $\{\cA(-1), \cX(-1), \cA, \cX\}$ is an exceptional sequence. To prove that it is full, we now consider the sequence $\{\cA(-1), \cA, \cY, \cX\}$.
By Lemma \ref{lem.mueXY}, we have
\[ L_2\{\cA(-1), \cA, \cY, \cX\} = \{\cA(-1), L_{\cA}(\cY), \cA, \cX \} = \{\cA(-1), \cX(-1), \cA, \cX \},  \]
so $\{\cA(-1), \cA, \cY, \cX\}$ is an exceptional sequence by Lemma \ref{lem.mue}.
Moreover
\begin{align*}
\langle \cA(-1), \cA,  \cY, \cX \rangle 
&\cong \langle \cA(-1), \cA,  \Phi\cD_{\rm Sg}^{\rm gr}(A) \rangle &&\textnormal{by Lemma \ref{lem.AmiXY}}\\
&\cong\cD^b(\tails A) &&\textnormal{by \cite[Theorem 2.5 (\rnum{1})]{O}},
\end{align*}
so $\{\cA(-1), \cA, \cY, \cX\}$ is a full exceptional sequence.
Hence $\{\cA(-1), \cX(-1), \cA, \cX\}$ is also a full exceptional sequence by Lemma \ref{lem.mue} again.
\end{proof}

\begin{theorem} \label{thm.qgh}
Let $E_{2i}:= \cA(i), E_{2i+1}:=\cX(i)$. Then $\{E_i\}_{i\in\ZZ}$
is a full geometric helix of period 4 for $\cD^b(\tails A)$.
\end{theorem}

\begin{proof}
Since $(-)_{\nu}$ commutes with shifts, if $\cM \in \{\cA(i), \cX(i)\}_{i\in \ZZ}$, then
$\cM\lotimes_{\cA} \omega_{\cA}^{-1} \cong  \cM_{\nu^{-1}}(2) \cong \cM(2)$
by Lemma \ref{lem.nu}, so $E_{i+4} \cong  E_i \lotimes_{\cA} \omega_{\cA}^{-1}$. 
To show that $\{E_i, E_{i+1}, E_{i+2}, E_{i+3}\}$ is a full exceptional sequence for every $i\in \ZZ$, it is enough to show that 
$\{\cA, \cX, \cA(1), \cX(1)\}$ and 
$\{\cX, \cA(1), \cX(1),\cA(2)\}$ are  full exceptional sequences. 
By Theorem \ref{thm.fe2}, $\{\cA, \cX, \cA(1), \cX(1)\}$ is a full exceptional sequence. 
We now show that $\{\cX, \cA(1), \cX(1),\cA(2)\}$ is a full exceptional sequence.
Similar to Theorem \ref{thm.fe2}, we see that $\{\cA(1), \cY(1), \cA(2), \cY(2)\}$ is a full exceptional sequence.
Since $L_{\cA(1)}(\cY(1)) = \cX$ and $L_{\cA(2)}(\cY(2)) = \cX(1)$ by Lemma \ref{lem.mueXY}, it follows that
\begin{align*}
\{ \cX, \cA(1), \cX(1), \cA(2) \}
&= \{L_{\cA(1)}(\cY(1)), \cA(1), L_{\cA(2)}(\cY(2)), \cA(2)\} \\
&= L_3 L_1 \{\cA(1), \cY(1), \cA(2), \cY(2)\}
\end{align*}
is a full exceptional sequence by Lemma \ref{lem.mue}.

By Lemmas \ref{lem.HE} (4), \ref{lem.Ext1}, \ref{lem.Ext2}, we have
\begin{align*}
&\Ext^q_\cA(\cA, \cX(i)) = 0 \quad (q\geq 1, i\geq 0), 
&&\Ext^q_\cA(\cA, \cA(i)) = 0 \quad (q\geq 1, i\geq 1),\\
&\Ext^q_\cA(\cX, \cA(i)) = 0 \quad (q\geq 1, i\geq 1),
&&\Ext^q_\cA(\cX, \cX(i)) = 0 \quad (q\geq 1, i\geq 1),
\end{align*}
so it follows that $\{E_i\}_{i\in\ZZ}$ is geometric.
\end{proof}

The following is an application of the main result of this paper.

\begin{theorem} \label{thm.sq}
For every standard smooth quadric surface $\cC$ in a quantum $\PP^3$, there exists a right noetherian AS-regular algebra $B$ over $kK_2$
of dimension 3 and of Gorenstein parameter $2$ such that $\cC\cong \tails B$ where $kK_2$ is the path algebra of the 2-Kronecker quiver $K_2$.   
\end{theorem}

\begin{proof}
Suppose that $\cC=\tails S/(f)$ 
where $S$ is a noetherian 4-dimensional quadratic AS-regular algebra over $k$ and $f\in S_2$ is a central regular element such that $A=S/(f)$ is a domain and a graded isolated singularity.
If $\cO := \cA\oplus \cX\in \tails A$ and $s:= (1) \in \Aut _k(\tails A)$, then $\{s^i\cO\}_{i\in\ZZ}$ is a full geometric relative helix of period 2 for $\cD^b(\tails A)$ by Theorem \ref{thm.qgh} and Lemma \ref{lem.lr}.
Since $A$ is a noetherian AS-Gorenstein graded isolated singularity of dimension 3 and $A\oplus X\in \CM^{\ZZ}(A)$ such that $\cO=\pi (A\oplus X)$,
it follows that $(\cO,s)$ is an ample algebraic pair for $\tails A$ by the proof of \cite[Theorem 2.5]{MU0}.
Since
$$ \End_{\cA}(\cO) = \begin{pmatrix} \End_{\cA}(\cA)
&\Hom_\cA(\cA, \cX)\\\Hom_\cA(\cX, \cA) &\End_{\cA}(\cX) \end{pmatrix}
=  \begin{pmatrix} k & k^2 \\ 0 &k \end{pmatrix}
= kK_2,$$
we see that $\tails A$ is equivalent to $\tails B$ for an AS-regular algebra $B=B(\tails A, \cO, s)$ over $kK_2$
of dimension $\gldim (\tails A)+1=3$ and of Gorenstein parameter $2$ by Theorem \ref{thm.main}. 
Since $\cO= \cA\oplus \cX$ is a noetherian object in $\tails A$ as a direct sum of two noetherian objects, 
$B=B(\tails A, \cO, s)$ is right noetherian by Theorem \ref{thm.main}.  
\end{proof}

\begin{example} Let $Q$ be a smooth quadric surface in $\PP^3$.  If $B=k\<x, y\>/(x^2y-yx^2, xy^2-y^2x)$, then $B$ is a right noetherian AS-regular algebra over $k$ of dimension 3 and of Gorenstein parameter $4$ such that $\coh Q\cong \tails B$, so
$$B^{[2]} = \bigoplus_{i \in \NN} \begin{pmatrix} B_{2i} & B_{2i+1} \\ B_{2i-1} &B_{2i} \end{pmatrix}$$ is a right noetherian AS-regular algebra over $kK_2$
of dimension 3 and of Gorenstein parameter $2$ by Corollary \ref {cor.qv} such that $\coh Q\cong \tails B\cong \tails B^{[2]}$.
\end{example} 

\section{Appendix}  

In this appendix, we will give an example of a non-standard algebra. 

\begin{example} Let $A=S/(f)$ where $S=k[x, y, z, w]$ such that $\deg x=\deg y=\deg z=\deg w=1$ and $f=xw-yz\in S_2$ so that $A$ is the homogeneous coordinate ring of a smooth quadric surface in $\PP^3$.  It is easy to see that $A$ is standard.  
In fact, if
$$M=\begin{pmatrix} x & y \\ z & w \end{pmatrix} \quad \textrm{and} \quad N=\begin{pmatrix} w & -y \\ -z & x \end{pmatrix},$$
then ${\mathbb M}=\{\Coker (\overline M\cdot), \Coker (\overline N\cdot)\}$.  Here we view $M$ as a matrix whose entries are in $S_1$ and $\overline M$ as a matrix whose entries are in $A_1$.  Since $MN=NM=fE$ in $S$, 
$$\begin{CD} 
\cdots @>\overline N\cdot >> A(-3)^2 @>\overline M\cdot >> A(-2)^2 @>\overline N\cdot >> A(-1)^2 @>\overline M\cdot>> A^2 \to \Coker (\overline M\cdot) \to 0, \\
\cdots @>\overline M\cdot >> A(-3)^2 @>\overline N\cdot >> A(-2)^2 @>\overline M\cdot >> A(-1)^2 @>\overline N\cdot>> A^2 \to \Coker (\overline N\cdot) \to 0
\end{CD}$$  
are the minimal free resolutions of $\Coker (\overline M\cdot), \Coker (\overline N\cdot)$ over $A$, so $A$ is standard.  

If $\s\in \GrAut S$ is a graded algebra automorphism of $S$ defined by $\s(x)=w, \s(y)=-y, \s(z)=-z, \s(w)=x$, then $\s(f)=f$, so $\s$ induces a graded algebra automorphism of $A$ such that $A^{\s}=S^{\s}/(f^{\s})$ where 
$$S^{\s}=k\<x, y, z, w\>/(xy+yw, xz+zw, x^2-w^2, yz-zy, yx+wy, zx+wz)$$
is a noetherian quadratic AS-regular algebra over $k$ of dimension 4 and $f^{\s}=x^2+yz\in S^{\s}$ is a central regular element.  Since $A$ is a domain and a graded isolated singularity, so is $A^{\s}$.  Since 
$M^2=N^2=f^{\s}E$ in $S^{\s}$, it follows that ${\mathbb M}=\{\Coker (\overline M\cdot), \Coker (\overline N\cdot)\}$, and 
$$\begin{CD} 
\cdots @>\overline M\cdot >> A^{\s}(-3)^2 @>\overline M\cdot >> A^{\s}(-2)^2 @>\overline M\cdot >> A^{\s}(-1)^2 @>\overline M\cdot>> (A^{\s})^2 \to \Coker (\overline M\cdot) \to 0 \\
\cdots @>\overline N\cdot >> A^{\s}(-3)^2 @>\overline N\cdot >> A^{\s}(-2)^2 @>\overline N\cdot >> A^{\s}(-1)^2 @>\overline N\cdot>> (A^{\s})^2 \to \Coker (\overline N\cdot) \to 0
\end{CD}$$  
are the minimal free resolutions of $\Coker (\overline M\cdot), \Coker (\overline N\cdot)$ over $A^{\s}$, so $A^{\s}$ is non-standard.  However, since $\tails A^{\s}\cong \tails A$, Theorem \ref{thm.sq} applies to this case.  
\end{example} 



\begin{thebibliography}{99}
\bibitem{AZ}
M. Artin and J. J. Zhang,
Noncommutative projective schemes,
\textit{Adv. Math.} \textbf{109} (1994), no. 2, 228--287.

\bibitem{BS}
P. Balmer and M. Schlichting,
Idempotent completion of triangulated categories,
\textit{J. Algebra} \textbf{236} (2001), no. 2, 819--834. 

\bibitem{B}
A. I. Bondal,
Representation of associative algebras and coherent sheaves,
\textit{Math. USSR-Izv.} \textbf{34} (1990), no. 1, 23--42.

\bibitem{BK}
A. I. Bondal and M. M. Kapranov,
Representable functors, Serre functors, and reconstructions,
\textit{Math. USSR-Izv.} \textbf{35} (1990), no. 3, 519--541.

\bibitem{Bu}
R.-O. Buchweitz,
Maximal Cohen-Macaulay modules and Tate cohomology
over Gorenstein rings,
unpublished manuscript (1985).

\bibitem{BH}
R.-O. Buchweitz and L. Hille,
Higher representation-infinite algebras from geometry,
Oberwolfach report \textbf{8} (2014), 466--469.

\bibitem{CS}
A. Canonacoa and P. Stellari, 
A tour about existence and uniqueness of dg enhancements and lifts,
\textit{J. Geom. Phys.} \textbf{122} (2017), 28--52.

\bibitem{CYZ}
X. W. Chen, Y. Ye and P. Zhang,
Algebras of derived dimension zero,
\textit{Comm. Algebra} \textbf{36} (2008), no. 1, 1--10. 


\bibitem{IT}
O. Iyama, and R. Takahashi, 
Tilting and cluster tilting for quotient singularities, 
\textit{Math. Ann.} \textbf{356} (2013), no. 3, 1065--1105.

\bibitem{Jcm}
P. J\o rgensen,
Finite Cohen-Macaulay type and smooth non-commutative schemes,
\textit{Canad. J. Math.} \textbf{60} (2008), no. 2, 379--390. 

\bibitem{Min}
H. Minamoto,
Ampleness of two-sided tilting complexes, 
\textit{Int. Math. Res. Not.} (2012), no. 1, 67--101. 

\bibitem{MM}
H. Minamoto and I. Mori,
The structure of AS-Gorenstein algebras,
\textit{Adv. Math.} \textbf{226} (2011), no. 5, 4061--4095. 

\bibitem{Mbc}
I. Mori,
B-construction and C-construction,
\textit{Comm. Algebra} \textbf{41} (2013), no. 6, 2071--2091. 

\bibitem{MU0}
I. Mori and K. Ueyama,
Ample group actions on AS-regular algebras and noncommutative graded isolated singularities,
{\it Trans. Amer. Math. Soc.} {\bf 368} (2016), no. 10, 7359--7383.

\bibitem{MU}
I. Mori and K. Ueyama,
Stable categories of graded maximal Cohen-Macaulay modules over noncommutative quotient singularities, 
\textit{Adv. Math.} \textbf{297} (2016), 54--92.


\bibitem{O}
D. Orlov, \textit{Derived categories of coherent sheaves and triangulated categories of singularities},  
Algebra, arithmetic, and geometry: in honor of Yu. I. Manin. Vol. II,  503--531, \textit{Progr. Math.}, \textbf{270}, 
Birkhauser Boston, Inc., Boston, MA, 2009.

\bibitem{P} 
A. Polishchuk,
Noncommutative proj and coherent algebras,
\textit{Math. Res. Lett.} \textbf{12} (2005), no. 1, 63--74. 

\bibitem{Se}
J.-P. Serre, Faisceaux alg\'ebriques coh\'erents,
\textit{Ann. of Math. (2)} \textbf{61} (1955), 197--278.

\bibitem{SV}
S. P. Smith and M. Van den Bergh,
Noncommutative quadric surfaces,
\textit{J. Noncommut. Geom.} \textbf{7} (2013), no. 3, 817--856.

\bibitem{U}
K. Ueyama, 
Graded maximal Cohen-Macaulay modules over noncommutative graded Gorenstein isolated singularities,
\textit{J. Algebra} \textbf{383} (2013), 85--103.

\bibitem{Ucm}
K. Ueyama,
Noncommutative graded algebras of finite Cohen-Macaulay representation type,
\textit{Proc. Amer. Math. Soc.} \textbf{143} (2015), no. 9, 3703--3715.

\bibitem{Uct} 
K. Ueyama, 
Cluster tilting modules and noncommutative projective schemes,
\textit{Pacific J. Math.} \textbf{289} (2017), no. 2, 449--468.

\bibitem{V}
M. Van den Bergh,
Existence theorems for dualizing complexes over non-commutative graded and filtered rings,
\textit{J. Algebra} \textbf{195} (1997), no. 2, 662--679.

\bibitem{Ye} 
A. Yekutieli, 
Dualizing complexes over noncommutative graded algebras,
\textit{J. Algebra}  \textbf{153}  (1992),  no. 1, 41--84.

\bibitem{Zh}
J. J. Zhang,
Twisted graded algebras and equivalences of graded categories,
\textit{Proc. Lond. Math. Soc. (3)} \textbf{72} (1996), no. 2, 281--311. 
\end{thebibliography}
\end{document}